\documentclass[10pt]{amsart}
\usepackage{amssymb}

\newtheorem{proposition}{Proposition}[section]
\newtheorem{lemma}[proposition]{Lemma}
\newtheorem{corollary}[proposition]{Corollary}
\newtheorem{theorem}[proposition]{Theorem}
\newtheorem{remark}[proposition]{Remark}

\theoremstyle{definition}

\newcommand{\selabel}[1]{\label{se:#1}}
\newcommand{\seref}[1]{Section~\ref{se:#1}}

\def\<{\leqslant}
\def\>{\geqslant}
\def\a{\alpha}
\def\b{\beta}
\def\d{\delta}
\def\g{\gamma}

\def\e{\varepsilon}

\def\s{\sigma}

\def\ti{\times}
\def\ot{\otimes}
\def\ra{\rightarrow}

\date{}

\begin{document}
\title{The Coalgebra Automorphism Group of Hopf algebra $k_q[x, x^{-1}, y]$}
\author{Hui-Xiang Chen}
\address{School of Mathematical Science, Yangzhou University,
Yangzhou 225002, China}
\email{hxchen@yzu.edu.cn}
\footnote{2010{\it Mathematics Subject Classification}. 16T15, 16W20}{}
\keywords{Automorphism, Hopf algebra, coalgebra, graded coalgebra}
\begin{abstract}
Let $k_q[x, x^{-1}, y]$ be the localization of the quantum plane $k_q[x, y]$ over
a field $k$, where $0\neq q\in k$. Then $k_q[x, x^{-1}, y]$ is a graded Hopf algebra,
which can be regarded as the non-negative part of the quantum enveloping algebra $U_q({\mathfrak sl}_2)$.
Under the assumption that $q$ is not a root of unity,
we investigate the coalgebra automorphism group of $k_q[x, x^{-1}, y]$. We describe the structures of the
graded coalgebra automorphism group and the coalgebra automorphism group of $k_q[x, x^{-1}, y]$,
respectively.
\end{abstract}

\maketitle
\section*{\bf Introduction}\selabel{0}

The automorphism group of a mathematical object is, in roughly speaking, the symmetry of the object.
To determine the automorphism group of a mathematical object is always fundamental in classification problem.
There are many papers concerning the automorphism groups of algebras.
Usually, it is very difficult to determine the automorphism group of an algebra. A well studied example
is the automorphism group of an incidence algebra, see \cite{Bac, Coel, Schar, Stan}. Andruskiewitsch and Dumas studied
the algebra automorphisms and Hopf algebra automorphisms of $U^+_q(\mathfrak{g})$,
the positive part of the quantum enveloping algebra of a simple complex finite dimensional Lie algebras
$\mathfrak{g}$ of type $A_2$ and $B_2$ in \cite{AndruDumas}. Li and Yu studied the algebra isomorphisms
and automorphisms of the quantum enveloping algebras $U_q({\mathfrak sl}_2)$ in \cite{LiYu}.
For more works on the algebra automorphisms  and Hopf algebra automorphisms
of $U^+_q({\mathfrak g})$ and $U_q({\mathfrak g})$, the reader
is directed to \cite{AlevCha, ChinMu, Lau1, Lau2, LauLop, Twie}.
On the other hand, quantum polynomial algebras are useful tools to study quantum groups, see
\cite{Arta, BrownGood, Demi, Manin}. Kirkman, Procesi and Small studied the automorphism group of
the quantum polynomial algebra $k_q[x, x^{-1}, y, y^{-1}]$ in \cite{KirPS}. Artamonov studied
the algebra automorphisms of the quantum polynomial algebra $k_q[x, x^{-1}, y]$ in \cite{Arta05}.
However, for the automorphism groups of coalgebras, the known examples are few in literature.
Ye studied the automorphism groups of path coalgebras in \cite{Ye}.
The present work aims to investigate the coalgebra automorphism group
of the Hopf algebra $k_q[x, x^{-1}, y]$, which can be regard as the non-negative part $U_q({\mathfrak sl}_2)^{\>0}$
of the quantum enveloping algebra $U_q({\mathfrak sl}_2)$.

In this paper, we investigate the coalgebra automorphism group of the Hopf algebra $k_q[x, x^{-1}, y]$
over a field $k$, where $0\neq q\in k$ and $q$ is not a root of unity.
In Section 1, we recall some basic definitions and notations, and make some preparations for the rest of the paper.
In Section 2, we first introduce the Hopf algebra $H:=k_q[x, x^{-1}, y]$, which is a graded pointed
Hopf algebra. Then we investigate the coalgebra automorphisms of $H$, and give some properties of the coalgebra
automorphisms of $H$. Let $Aut_c(H)$ be the coalgebra automorphism group of $H$, and $Aut_0(H)$ the subgroup of
$Aut_c(H)$ consisting of all the coalgebra automorphisms $\phi$ satisfying $\phi(1)=1$. For a $\phi\in Aut_c(H)$,
it is shown that $\phi\in Aut_0(H)$ if and only if the restriction of $\phi$ on the coradical $H_0$ of $H$
is the identity. Then we construct a subgroup $\Theta$ of $Aut_c(H)$ and show that $\Theta$ is isomorphic
to the additive group $\mathbb Z$ of all integers. We also show that $Aut_0(H)$ is a normal subgroup of $Aut_c(H)$,
and that $Aut_c(H)$ is the internal semidirect product of $Aut_0(H)$ and $\Theta$.
Let $Aut^{gr}_c(H)$ be the graded coalgebra automorphism group of $H$ and
$Aut^{gr}_0(H)=Aut^{gr}_c(H)\cap Aut_0(H)$. Then $Aut^{gr}_c(H)$ is the internal semidirect product of $Aut^{gr}_0(H)$
and $\Theta$. We show that $Aut^{gr}_0(H)$ is isomorphic to $(k^{\times})^{\mathbb Z}$, the direct product group of
$\mathbb Z$-copies of the multiplicative group $k^{\times}$ of all nonzero scales in the ground field $k$.
Finally, we show that $Aut^{gr}_c(H)$ is isomorphic to a semidirect product group
$(k^{\times})^{\mathbb Z}\rtimes\mathbb{Z}$.
In Section 3, we investigate the structure of the coalgebra automorphism group $Aut_c(H)$.
We first construct a family of normal subgroups of $Aut_c(H)$:
$$Aut_0(H)\supseteq Aut_*(H)\supseteq Aut_1(H)\supseteq Aut_2(H)\supseteq\cdots.$$
It is shown that $Aut_c(H)$ (resp. $Aut_0(H)$) is the internal semidirect product of $Aut_*(H)$ and
$Aut^{gr}_c(H)$ (resp. $Aut^{gr}_0(H)$), and that $Aut_*(H)/Aut_1(H)\cong Aut_{s-1}(H)/Aut_{s}(H)\cong
k^{\mathbb Z}$, the direct product group of $\mathbb Z$-copies of the additive group $k$, for all $s\>2$.
Then we show that the family of quotient groups $\{Aut_*(H)/Aut_i(H)\}_{i\in I}$ forms an inverse system
of groups with the index set $I$ of all positive integers, and that $Aut_*(H)$ is isomorphic to the inverse limit
$\underleftarrow{\rm lim}(Aut_*(H)/Aut_i(H))$.
Let $G_{\infty}=(k^{\mathbb Z})^I$ be the Cartesian product set of $I$-copies of $k^{\mathbb Z}$.
Using a recursive method, we define a group structure on $G_{\infty}$, and show that
$Aut_*(H)\cong G_{\infty}$. Finally, we show that $Aut_c(H)$ (resp. $Aut_0(H)$) is isomorphic to
a semidirect product group $G_{\infty}\rtimes((k^{\times})^{\mathbb Z}\rtimes\mathbb{Z})$
(resp. $G_{\infty}\rtimes(k^{\times})^{\mathbb Z}$).

\section{\bf Preliminaries}\selabel{1}

Throughout, let $k$ be an arbitrary field. Unless
otherwise stated, all algebras, coalgebras and Hopf algebras are
defined over $k$; linear and $\otimes$ stand for $k$-linear and $\otimes_k$,
respectively. Let $\mathbb Z$ denote the set of all integers,
and $\mathbb N$ denote the set of all non-negative integers.
Let $k^{\ti}$ denote the multiplicative group of all nonzero elements in the field $k$.
For the theory of Hopf algebras and quantum groups, we refer to \cite{Ka, Maj, Mon, Sw}.

Let $C$ be a coalgebra, and $G(C)$ the set of all group-like elements in $C$.
For $g, h\in G(C)$, an element $c\in C$ is called a $(g, h)$-primitive element
if $\Delta(c)=c\ot g+h\ot c$. Let $P_{g, h}(C)$ denote the set of all $(g, h)$-primitives
in $C$. Then $P_{g, h}(C)$ is a subspace of $C$.

%

A vector space with a designated direct sum decomposition $V=\bigoplus_{n=0}^{\infty}V(n)$
of subspaces is a graded vector space. Suppose that  $U=\bigoplus_{n=0}^{\infty}U(n)$ and
$V=\bigoplus_{n=0}^{\infty}V(n)$ are graded vector spaces. Then a linear map
$f: U\ra V$ is a graded map if $f(U(n))\subseteq V(n)$ for
all $n\geqslant 0$.

An algebra $A$ is a graded algebra if $A=\bigoplus_{n=0}^{\infty}A(n)$
is a graded vector space such that $1\in A(0)$ and $A(n)A(m)\subseteq A(n+m)$
for all $n, m\geqslant 0$. A coalgebra $C$ is a graded coalgebra if
$C=\bigoplus_{n=0}^{\infty}C(n)$ is a graded vector space such that
$\varepsilon(C(n))=0$ for all $n>0$ and
$\Delta(C(n))\subseteq\sum_{i=0}^nC(i)\ot C(n-i)$ for all $n\geqslant 0$.

A bialgebra $H$ is a graded bialgebra if  $H=\bigoplus_{n=0}^{\infty}H(n)$
is a graded vector space, which gives $H$ a graded algebra structure and a graded coalgebra
structure. A Hopf algebra $H$ is a graded Hopf algebra if $H=\bigoplus_{n=0}^{\infty}H(n)$
is a graded bialgebra such that the antipode $S$ is a graded map.

Let $C=\bigoplus_{n=0}^{\infty}C(n)$ and $D=\bigoplus_{n=0}^{\infty}D(n)$
be two graded coalgebras. A coalgebra map (or isomorphism) $f: C\ra D$ is called a graded coalgebra
map (or isomorphism) if $f$ is a graded map.

Let $0\not=q\in k$. For any integer $n>0$, set
$(n)_q=1+q+\cdots +q^{n-1}$.
Observe that $(n)_q=n$ when $q=1$, and
$$
(n)_q=\frac{q^n-1}{q-1}
$$
when $q\not= 1$.
Define the $q$-factorial of $n$ by
$(0)!_q=1$ and
$(n)!_q=(n)_q(n-1)_q\cdots (1)_q$ for $n>0$.
Note that $(n)!_q=n!$ when $q=1$, and
$$
(n)!_q=
\frac{(q^n-1)(q^{n-1}-1)\cdots (q-1)}{(q-1)^n}
$$
when $n>0$ and $q\not= 1$.
The  $q$-binomial coefficients
$
\left(\begin{array}{c}
n\\
i\\
\end{array}\right)_q
$
is defined inductively as follows for $0\leqslant i\leqslant n$:
$$
\left(\begin{array}{c}
n\\
0\\
\end{array}\right)_q
=1=
\left(\begin{array}{c}
n\\
n\\
\end{array}\right)_q
\quad\quad
\mbox{ for } n\geqslant 0,$$
$$
\left(\begin{array}{c}
n\\
i\\
\end{array}\right)_q
=
q^i
\left(\begin{array}{c}
n-1\\
i\\
\end{array}\right)_q
+
\left(\begin{array}{c}
n-1\\
i-1\\
\end{array}\right)_q
\quad \quad
\mbox{ for } 0< i< n.$$
It is well-known that
$
\left(\begin{array}{c}
n\\
i\\
\end{array}\right)_q$
is a polynomial in $q$ with integer coefficients and with value at $q=1$
equal to the usual binomial coefficients
$
\left(\begin{array}{c}
n\\
i\\
\end{array}\right)$, and that
$$
\left(\begin{array}{c}
n\\
i\\
\end{array}\right)_q
=\frac{(n)!_q}{(i)!_q(n-i)!_q}
$$
when
$(n-1)!_q\not = 0$ and $0<i<n$
(see \cite[p.74]{Ka}).

Let $k^{\mathbb Z}$ (resp. $(k^{\times})^{\mathbb Z}$) be the Cartesian product set of
$\mathbb Z$-copies of $k$ (resp. $k^{\times}$). Then $(k^{\times})^{\mathbb Z}\subset k^{\mathbb Z}$.
For any $\a=(\a_n)_{n\in\mathbb Z}$, $\b=(\b_n)_{n\in\mathbb Z}\in k^{\mathbb Z}$, define
$\a+\b=(\a_n+\b_n)_{n\in\mathbb Z}$ and $\a\b=(\a_n\b_n)_{n\in\mathbb Z}$, i.e.,
$(\a+\b)_n=\a_n+\b_n$ and $(\a\b)_n=\a_n\b_n$ for all $n\in\mathbb Z$. Then $k^{\mathbb Z}$ is an additive
group as the direct product of $\mathbb Z$-copies of the additive group $k$,
and $(k^{\times})^{\mathbb Z}$ is a multiplicative group as the direct product of $\mathbb Z$-copies
of the multiplicative group $k^{\times}$.

For $\a=(\a_n)_{n\in\mathbb Z}\in k^{\mathbb Z}$ (or $(k^{\times})^{\mathbb Z}$), $a\in k$ (or $k^{\times}$)
and $r\in\mathbb Z$, define $a\a=((a\a)_n)_{n\in\mathbb Z}$ and $\a[r]=(\a[r]_n)_{n\in\mathbb Z}$ by
$(a\a)_n=a\a_n$ and $\a[r]_n=\a_{n+r}$ for all $n\in\mathbb Z$, respectively.

\section{\bf The Graded Coalgebra Automorphisms of Hopf algebra $k_q[x, x^{-1}, y]$}\selabel{2}

Let $0\neq q\in k$. $k_q[x, x^{-1}, y]$ is an algebra over $k$ generated by
$x$, $x^{-1}$ and $y$ subject to the relations
$$xx^{-1}=1,\ x^{-1}x=1,\ yx=qxy.$$
Then $k_q[x, x^{-1}, y]$ is a Hopf algebra with the coalgebra structure and antipode
$S$ given by
$$\begin{array}{lll}
\Delta(x)=x\ot x,& \Delta(x^{-1})=x^{-1}\ot x^{-1},& \Delta(y)=y\ot x+1\ot y,\\
\varepsilon(x)=1,& \varepsilon(x^{-1})=1,& \varepsilon(y)=0,\\
S(x)=x^{-1},& S(x^{-1})=x,& S(y)=-yx^{-1}=-q^{-1}x^{-1}y.\\
\end{array}$$
For the details, the reader is directed to \cite{chen98}. In \cite{Rad03},
the Hopf algebra $k_q[x, x^{-1}, y]$ is denoted by $H_q$.
When $q$ is a primitive $n$-th root of unity for some $n\geqslant 2$,
the ideal $I$ of $k_q[x, x^{-1}, y]$ generated by $x^n-1$ and $y^n$ is a Hopf ideal.
In this case, the quotient Hopf algebra $k_q[x, x^{-1}, y]/I$ is exactly the Taft Hopf algebra,
see \cite{Rad03, Ta}. If $q^2\neq 1$, the $k_{q^{-2}}[x, x^{-1}, y]$ is isomorphic to
$U_q({\mathfrak sl}_2)^{\>0}$, the non-negative part of the quantum enveloping algebra
$U_q({\mathfrak sl}_2)$ as a Hopf algebras, see \cite{HuZhang} and \cite{Zhang-Chen08}.

Throughout the following, assume that $q$ is not a root of unity,
and denote $H_q=k_q[x, x^{-1}, y]$ by $H$ for simplicity. In this case,
$(n)!_q\neq 0$ for all $n\>1$, and so $\left (\begin{array}{c}n\\ i \end{array} \right )_{q}\neq0$
for all $0\<i\<n$.

The following lemma is known, see \cite{chen98, Ka, Rad03}.

\begin{lemma}\label{basic}
The following statements hold for $H$.\\
{\rm (1)} $\{x^ny^m|n\in{\mathbb Z}, m\in{\mathbb N}\}$ is a $k$-basis of $H$.\\
{\rm (2)} $\Delta(x^ny^m)=\sum\limits_{i=0}^m
\left(\begin{array}{c}
m\\
i\\
\end{array}\right)_qx^ny^i\ot x^{n+i}y^{m-i}$, $n\in\mathbb Z$, $m\in\mathbb N$.\\
{\rm (3)} $H$ is a pointed Hopf algebra with $H_0=kG(H)$ and $G(H)=\{x^n|n\in\mathbb Z\}$, an infinite
cyclic group.\\
{\rm (4)} $H=\bigoplus_{n=0}^{\infty}H(n)$ is a graded Hopf algebra, where $H(n)=H_0y^n$
for all $n\geqslant 0$.
\end{lemma}

\begin{lemma}\label{auto01}
Let $\phi: H\ra H$ be a coalgebra automorphism. Then there is an integer $r$ such that for any
$n\in\mathbb{Z}$,
$$\phi(x^n)=x^{n+r},\ \phi(x^ny)=\a_n x^{n+r}y+\b_n(x^{n+r+1}-x^{n+r})$$
for some $\a_n\in k^{\times}$ and $\b_n\in k$.
\end{lemma}

\begin{proof}
Since $\phi$ is a coalgebra automorphism of $H$ and $G(H)=\{x^n|n\in\mathbb{Z}\}$,
the restriction of $\phi$ on $G(H)$ gives rise to
a bijection from $G(H)$ onto itself. Hence there is a permutation $\theta$ of $\mathbb Z$
such that $\phi(x^n)=x^{\theta(n)}$ for all $n\in\mathbb{Z}$.

Now let $n\in\mathbb Z$ be an arbitrary fixed integer. Assume that
$\phi(x^ny)=\sum\limits_{s\in\mathbb{Z}, m\in\mathbb{N}}\mu_{s,m}x^sy^m$,
where $\mu_{s,m}\in k$ and almost all $\mu_{s,m}=0$. Then
$\Delta\phi(x^ny)=(\phi\ot\phi)\Delta(x^ny)$ since $\phi$ is a coalgebra map.
By Lemma \ref{basic}, we have
$$\begin{array}{rcl}
\Delta\phi(x^ny)&=&\sum\limits_{s\in\mathbb{Z}, m\in\mathbb{N}}\mu_{s,m}\Delta(x^sy^m)\\
&=&\sum\limits_{s\in\mathbb{Z}, m\in\mathbb{N}}\mu_{s,m}\sum\limits_{i=0}^m
\left(\begin{array}{c}
m\\
i\\
\end{array}\right)_qx^sy^i\ot x^{s+i}y^{m-i}\\
\end{array}$$
and
$$\begin{array}{rcl}
(\phi\ot\phi)\Delta(x^ny)&=&(\phi\ot\phi)(x^n\ot x^ny+x^ny\ot x^{n+1})\\
&=&\phi(x^n)\ot\phi(x^ny)+\phi(x^ny)\ot\phi(x^{n+1})\\
&=&\sum\limits_{s\in\mathbb{Z}, m\in\mathbb{N}}\mu_{s,m}x^{\theta(n)}\ot x^sy^m
+\sum\limits_{s\in\mathbb{Z}, m\in\mathbb{N}}\mu_{s,m}x^sy^m\ot x^{\theta(n+1)}.\\
\end{array}$$
It follows that
$$\begin{array}{rl}
&\sum\limits_{s\in\mathbb{Z}, m\in\mathbb{N}}\mu_{s,m}\sum\limits_{i=0}^m
\left(\begin{array}{c}
m\\
i\\
\end{array}\right)_qx^sy^i\ot x^{s+i}y^{m-i}\\
=&\sum\limits_{s\in\mathbb{Z}, m\in\mathbb{N}}\mu_{s,m}x^{\theta(n)}\ot x^sy^m
+\sum\limits_{s\in\mathbb{Z}, m\in\mathbb{N}}\mu_{s,m}x^sy^m\ot x^{\theta(n+1)}.\\
\end{array}$$
When $m>1$, one may choose an integer $i$ with $0<i<m$. Then for any $s\in\mathbb Z$,
by comparing the coefficients of the item $x^sy^i\ot x^{s+i}y^{m-i}$ of the both sides
of the above equation, one gets that
$\mu_{s,m}\left(\begin{array}{c}
m\\
i\\
\end{array}\right)_q=0$, and so $\mu_{s,m}=0$ since
$\left(\begin{array}{c}
m\\
i\\
\end{array}\right)_q\neq 0$.
Thus, $\phi(x^ny)=\sum_{s\in\mathbb{Z}}\sum_{m=0}^1\mu_{s,m}x^sy^m$,
and the above equation becomes the following one
$$\begin{array}{rl}
&\sum\limits_{s\in\mathbb{Z}}\mu_{s,0}x^s\ot x^s
+\sum\limits_{s\in\mathbb{Z}}\mu_{s,1}(x^s\ot x^sy+x^sy\ot x^{s+1})\\
\hspace{2cm}=&\sum\limits_{s\in\mathbb{Z}}(\mu_{s,0}x^{\theta(n)}\ot x^s+\mu_{s,1}x^{\theta(n)}\ot x^sy)\\
&+\sum\limits_{s\in\mathbb{Z}}(\mu_{s,0}x^s\ot x^{\theta(n+1)}
+\mu_{s,1}x^sy\ot x^{\theta(n+1)}).\hspace{2.2cm}(2.1)\\
\end{array}$$
When $s\neq\theta(n)$, by comparing the coefficients of the item
$x^{\theta(n)}\ot x^sy$ of the both sides of Eq.(2.1),
one gets that $\mu_{s, 1}=0$.
When $s\neq\theta(n+1)-1$, by comparing the coefficients of the item
$x^sy\ot x^{\theta(s+1)}$ of the both sides of Eq.(2.1),
one gets that $\mu_{s, 1}=0$. Thus, if $\theta(n+1)-1\neq\theta(n)$,
then $\mu_{s,1}=0$ for all $s\in\mathbb Z$. In this case,
$\phi(x^ny)\in kG(H)=H_0$, but $x^ny\notin H_0$. This contradicts to the hypothesis
that $\phi$ is a coalgebra automorphism of $H$. It follows that
$\theta(n+1)-1=\theta(n)$, i.e., $\theta(n+1)=\theta(n)+1$.
When $s\neq\theta(n)$ and $s\neq\theta(n+1)=\theta(n)+1$, by comparing the coefficients
of the item $x^s\ot x^s$ of the both sides of Eq.(2.1), one gets
that $\mu_{s,0}=0$. Summarizing the above discussion, we have
$$\phi(x^ny)=\mu_{\theta(n),0}x^{\theta(n)}+\mu_{\theta(n)+1,0}x^{\theta(n)+1}
+\mu_{\theta(n),1}x^{\theta(n)}y.$$
Since $\phi$ is a coalgebra map, we have $\varepsilon\phi(x^ny)=\varepsilon(x^ny)=0$,
which implies that $\mu_{\theta(n),0}+\mu_{\theta(n)+1,0}=0$, that is,
$\mu_{\theta(n),0}=-\mu_{\theta(n)+1,0}$.

Let $r=\theta(0)$. Since $\theta(n+1)=\theta(n)+1$ for all $n\in\mathbb Z$,
one knows that $\theta(n)=\theta(0)+n=r+n$ for all $n\in\mathbb Z$.
Therefore, we have
$$\phi(x^n)=x^{n+r},\ \phi(x^ny)=\mu_{n+r,1}x^{n+r}y+\mu_{n+r+1,0}(x^{n+r+1}-x^{n+r})$$
for all $n\in\mathbb Z$. Since $x^ny\notin H_0$, $\phi(x^ny)\notin H_0$, which implies
$\mu_{n+r, 1}\neq 0$.
\end{proof}

\begin{lemma}\label{Zpart}
Let $r\in\mathbb Z$. Define a linear map $\theta_r: H\ra H$ by
$$\theta_r(x^ny^m)=x^{n+r}y^m,\ \ n\in\mathbb{Z},\ m\in\mathbb{N}.$$
Then $\theta_r$ is a coalgebra automorphism of $H$.
\end{lemma}

\begin{proof}
It follows from a straightforward verification by using Lemma \ref{basic}(2).
\end{proof}

Let $Aut_c(H)$ denote the coalgebra automorphism group of $H$,
that is, $Aut_c(H)$ is the group consisting of all coalgebra automorphisms of $H$
with the composition as its multiplication. Let $\Theta=\{\theta_r|r\in\mathbb{Z}\}$.
Then $\Theta$ is a subgroup of $Aut_c(H)$, and $\theta_r\theta_t=\theta_{r+t}$
for all $r, t\in\mathbb{Z}$. Let $Aut_0(H)=\{\phi\in Aut_c(H)|\phi(1)=1\}$.
Obviously, $Aut_0(H)$ is also a subgroup of $Aut_c(H)$. Let $\phi\in Aut_c(H)$.
Then it follows from Lemma \ref{auto01} that $\phi\in Aut_0(H)$ if and only if
$\phi(x^n)=x^n$ for all $n\in\mathbb{Z}$.

\begin{proposition}\label{Autosemi}
$Aut_0(H)$ is a normal subgroup of $Aut_c(H)$. Moreover, $Aut_c(H)$ is the internal
semidirect product of $Aut_0(H)$ and $\Theta$.
\end{proposition}

\begin{proof}
Let $\phi\in Aut_c(H)$ and $\psi\in Aut_0(H)$. Then by Lemma \ref{auto01}, there exists
an integer $r\in\mathbb{Z}$ such that $\phi(x^n)=x^{n+r}$ for all $n\in\mathbb{Z}$.
Hence $\phi^{-1}(x^n)=x^{n-r}$ for all $n\in\mathbb{Z}$. Thus, we have
$(\phi\psi\phi^{-1})(1)=(\phi\psi)(x^{-r})=\phi(x^{-r})=1$, and so
$\phi\psi\phi^{-1}\in Aut_0(H)$. This shows that $Aut_0(H)$ is a normal subgroup of $Aut_c(H)$.
Furthermore, we have $(\phi\theta_{-r})(1)=\phi(x^{-r})=1$. Hence $\phi\theta_{-r}\in Aut_0(H)$,
and so $\phi=(\phi\theta_{-r})\theta_r\in Aut_0(H)\Theta$. It follows that
$Aut_c(H)=Aut_0(H)\Theta$. Similarly, one can show that $Aut_c(H)=\Theta Aut_0(H)$.
Obviously, $Aut_0(H)\cap\Theta=\{{\rm id}\}$, where id denotes the identity map on $H$,
i.e., the identity element of the group $Aut_c(H)$. Thus, we have shown that
$Aut_c(H)$ is the internal semidirect product of $Aut_0(H)$ and $\Theta$.
\end{proof}

From Proposition \ref{Autosemi}, one gets the following corollary.

\begin{corollary}
There is a group isomorphism $Aut_c(H)/Aut_0(H)\cong\Theta$.
\end{corollary}

From Lemma \ref{basic}, $H_0$ is a Hopf subalgebra of $H$. Moreover,
$H$ is a free left $H_0$-module with an $H_0$-basis $\{y^m|m\geqslant 0\}$.
Hence $H\ot H$ is a free left $H_0\ot H_0$-module, and $\{y^n\ot y^m|n, m\geqslant 0\}$
is an $H_0\ot H_0$-basis of $H\ot H$.

\begin{lemma}\label{primi}
$P_{x,1}(H)=ky+k(x-1)$ and $P_{x^m,1}(H)=k(x^m-1)$ if $m\neq 1$.
\end{lemma}

\begin{proof}
Let $m\in\mathbb{Z}$. Since $H=\bigoplus_{n=0}^{\infty}H(n)$ is a graded coalgebra,
we have
$$P_{x^m,1}(H)=\bigoplus\limits_{n=0}^{\infty}(P_{x^m,1}(H)\cap H(n)).$$
It is easy to check that $P_{x^m,1}(H)\cap H(0)=P_{x^m,1}(H)\cap H_0=k(x^m-1)$.
Now let $n>0$ and assume that there exists a nonzero element $h$ in $P_{x^m,1}(H)\cap H(n)$. Then we have
$\Delta(h)=h\ot x^m+1\ot h$ and $h=ay^n$ for some $0\neq a\in H_0$.
By Lemma \ref{basic}, we have
$$\Delta(h)=\sum\limits_{i=0}^n
\left(\begin{array}{c}
n\\
i\\
\end{array}\right)_q\Delta(a)(y^i\ot x^iy^{n-i})
=\sum\limits_{i=0}^n
\left(\begin{array}{c}
n\\
i\\
\end{array}\right)_q\Delta(a)(1\ot x^i)(y^i\ot y^{n-i})$$
and
$$h\ot x^m+1\ot h=ay^n\ot x^m+1\ot ay^n=(a\ot x^m)(y^n\ot 1)+(1\ot a)(1\ot y^n).$$
It follows that
$$\sum\limits_{i=0}^n
\left(\begin{array}{c}
n\\
i\\
\end{array}\right)_q\Delta(a)(1\ot x^i)(y^i\ot y^{n-i})
=(a\ot x^m)(y^n\ot 1)+(1\ot a)(1\ot y^n).$$
We have already known that $H\ot H$ is a free left $H_0\ot H_0$-module with a basis
$\{y^s\ot y^t|s, t\geqslant 0\}$. If $n>1$, one can choose an integer $i$ with
$0<i<n$. Then, by comparing the coefficients of the basis element $y^i\ot y^{n-i}$
of the both sides of the above equation, one finds that
$\left(\begin{array}{c}
n\\
i\\
\end{array}\right)_q\Delta(a)(1\ot x^i)=0$. Since $\left(\begin{array}{c}
n\\
i\\
\end{array}\right)_q\neq 0$, $\Delta(a)(1\ot x^i)=0$, and so $\Delta(a)=0$.
This implies that $a=0$ since $\Delta$ is injective, a contradiction.
Therefore, $n=1$ and $h=ay$. Thus, the above equation becomes
$$\Delta(a)(1\ot x)(y\ot 1)+\Delta(a)(1\ot y)
=(a\ot x^m)(y\ot 1)+(1\ot a)(1\ot y).$$
It follows that $\Delta(a)(1\ot x)=a\ot x^m$ and $\Delta(a)=1\ot a$.
From $\Delta(a)=1\ot a$, one knows that $a\in k$. Since $a\neq0$,
it follows from $\Delta(a)(1\ot x)=a\ot x^m$ that $m=1$. This completes the proof.
\end{proof}

\begin{lemma}\label{inclu}
Let $\phi\in Aut_0(H)$. Then $\phi(H(m))\subseteq\sum_{i=0}^mH(i)$
for all $m\geqslant 0$.
\end{lemma}

\begin{proof}
Note that $H(m)=H_0y^m$ for all $m\geqslant 0$, and that $H_0$ has a $k$-basis
$\{x^n|n\in\mathbb{Z}\}$. Hence we only need to show that
$\phi(x^ny^m)\subseteq\sum_{i=0}^mH(i)$ for all $n\in\mathbb{Z}$ and $m\in\mathbb{N}$.
When $m=0$ or $m=1$, it follows from Lemma \ref{auto01} that
$\phi(x^ny^m)\subseteq\sum_{i=0}^mH(i)$ for all $n\in\mathbb{Z}$.
Now let $m>1$ and assume $\phi(x^ny^s)\subseteq\sum_{i=0}^sH(i)$ for all $n\in\mathbb{Z}$
and $0\leqslant s<m$. Let $n\in\mathbb{Z}$. Then by Lemma \ref{basic}, we may assume
that $\phi(x^ny^m)=\sum_{i=0}^la_iy^i$, where $a_i\in H_0$, $0\leqslant i\leqslant l$,
and $a_l\neq 0$. It is enough to show that $l\leqslant m$.
Suppose $l>m$. Then $\Delta\phi(x^ny^m)=(\phi\ot\phi)\Delta(x^ny^m)$
since $\phi$ is a coalgebra map.
Now we have
$\Delta\phi(x^ny^m)=\Delta(\sum_{i=0}^la_iy^i)=\sum_{i=0}^l\Delta(a_iy^i)$
and
$$\begin{array}{rcl}
(\phi\ot\phi)\Delta(x^ny^m)&=&(\phi\ot\phi)(\sum\limits_{i=0}^m
\left(\begin{array}{c}
m\\
i\\
\end{array}\right)_qx^ny^i\ot x^{n+i}y^{m-i})\\
&=&\sum\limits_{i=1}^{m-1}
\left(\begin{array}{c}
m\\
i\\
\end{array}\right)_q\phi(x^ny^i)\ot\phi(x^{n+i}y^{m-i})\\
&&+\phi(x^n)\ot\phi(x^ny^m)+\phi(x^ny^m)\ot\phi(x^{n+m})\\
&=&\sum\limits_{i=1}^{m-1}
\left(\begin{array}{c}
m\\
i\\
\end{array}\right)_q\phi(x^ny^i)\ot\phi(x^{n+i}y^{m-i})\\
&&+\sum\limits_{i=0}^l(x^n\ot a_iy^i+a_iy^i\ot x^{n+m}).\\
\end{array}$$
Therefore, we have
$$\begin{array}{rcl}
\hspace{1.7cm}\sum\limits_{i=0}^l\Delta(a_iy^i)&=&\sum\limits_{i=1}^{m-1}
\left(\begin{array}{c}
m\\
i\\
\end{array}\right)_q\phi(x^ny^i)\ot\phi(x^{n+i}y^{m-i})\\
&&+\sum\limits_{i=0}^l(x^n\ot a_iy^i+a_iy^i\ot x^{n+m}).\hspace{2.5cm}(2.2)\\
\end{array}$$
Since $H$ is a graded coalgebra, so is $H\ot H$ with the grading given by
$(H\ot H)(n)=\sum_{i=0}^nH(i)\ot H(n-i)$ for all $n\geqslant 0$.
By the induction hypothesis, one knows that $\sum\limits_{i=1}^{m-1}
\left(\begin{array}{c}
m\\
i\\
\end{array}\right)_q\phi(x^ny^i)\ot\phi(x^{n+i}y^{m-i})\subseteq
\sum\limits_{i=0}^m(H\ot H)(i)$. From the definition of graded coalgebras,
the comultiplication $\Delta$ of a graded coalgebra is a graded map.
Comparing the homogeneous components of degree $l$ of the both sides of Eq.(2.2),
one finds that
$$\Delta(a_ly^l)=x^n\ot a_ly^l+a_ly^l\ot x^{n+m}=a_ly^l\ot x^{n+m}+x^n\ot a_ly^l.$$
Hence $a_ly^l\in P_{x^{n+m},x^n}(H)$, and so $x^{-n}a_ly^l\in P_{x^m,1}(H)$.
However, $0\neq x^{-n}a_ly^l\in H(l)$ and $l>m>1$, which contradicts to Lemma \ref{primi}.
This completes the proof.
\end{proof}

\begin{lemma}\label{notinclu}
Let $\phi\in Aut_0(H)$. Then $\phi(H(m))\nsubseteq\sum_{i=0}^{m-1}H(i)$
for all $m\geqslant 1$.
\end{lemma}

\begin{proof}
Suppose that there is an $m\geqslant 1$ such that $\phi(H(m))\subseteq\sum_{i=0}^{m-1}H(i)$.
Since $\phi\in Aut_0(H)$, $\phi^{-1}\in Aut_0(H)$. By lemma \ref{inclu}, we have
$\phi^{-1}(\sum_{i=0}^{m-1}H(i))\subseteq\sum_{i=0}^{m-1}H(i)$.
It follows that $H(m)=(\phi^{-1}\phi)(H(m))=\phi^{-1}(\phi(H(m)))
\subseteq\phi^{-1}(\sum_{i=0}^{m-1}H(i))\subseteq\sum_{i=0}^{m-1}H(i)$,
which is impossible. This completes the proof.
\end{proof}

\begin{lemma}\label{firstitem}
Let $\phi\in Aut_0(H)$. Then for any $n\in\mathbb{Z}$ and $m\geqslant 1$,
$\phi(x^ny^m)=\a_{n, m}x^ny^m+h_{n,m}$ for some $\a_{n,m}\in k^{\times}$ and $h_{n,m}\in\sum_{i=0}^{m-1}H(i)$.
Moreover, if $\a_{n,1}=1$ for all $n\in\mathbb{Z}$, then
$\a_{n,m}=1$ for all $n\in\mathbb{Z}$ and $m\geqslant 1$.
\end{lemma}

\begin{proof}
For any $n\in\mathbb{Z}$ and $m\geqslant 0$, by Lemmas \ref{auto01}, \ref{inclu}
and \ref{notinclu}, one may assume that $\phi(x^ny^m)=\sum_{i=0}^ma_{n,m,i}y^i$,
where $a_{n,m,i}\in H_0$ with $a_{n,m,m}\neq 0$. It follows from Lemma \ref{auto01}
that $a_{n,0,0}=x^n$ and $a_{n,1,1}=\a_{n,1} x^n$ for some $\a_{n,1}\in k^{\times}$.
Now let us consider the case of $m>1$. Since $\phi$ is a coalgebra map,
$\Delta\phi(x^ny^m)=(\phi\ot\phi)\Delta(x^ny^m)$. By Lemma \ref{basic}, we have
$$\begin{array}{rcl}
\Delta\phi(x^ny^m)&=&\sum\limits_{i=0}^m\Delta(a_{n,m,i}y^i)\\
&=&\sum\limits_{i=0}^m\Delta(a_{n,m,i})(\sum\limits_{j=0}^i
\left(\begin{array}{c}
i\\
j\\
\end{array}\right)_qy^j\ot x^jy^{i-j})\\
&=&\sum\limits_{i=0}^m\sum\limits_{j=0}^i
\left(\begin{array}{c}
i\\
j\\
\end{array}\right)_q\Delta(a_{n,m,i})(1\ot x^j)(y^j\ot y^{i-j})\\
\end{array}$$
and
$$\begin{array}{rcl}
(\phi\ot\phi)\Delta(x^ny^m)&=&(\phi\ot\phi)(\sum\limits_{i=0}^m
\left(\begin{array}{c}
m\\
i\\
\end{array}\right)_qx^ny^i\ot x^{n+i}y^{m-i})\\
&=&\sum\limits_{i=0}^m
\left(\begin{array}{c}
m\\
i\\
\end{array}\right)_q\phi(x^ny^i)\ot\phi(x^{n+i}y^{m-i})\\
&=&\sum\limits_{i=0}^m
\left(\begin{array}{c}
m\\
i\\
\end{array}\right)_q\sum\limits_{j=0}^i\sum\limits_{l=0}^{m-i}a_{n,i,j}y^j\ot a_{n+i,m-i,l}y^l\\
&=&\sum\limits_{i=0}^m\sum\limits_{j=0}^i\sum\limits_{l=0}^{m-i}
\left(\begin{array}{c}
m\\
i\\
\end{array}\right)_q(a_{n,i,j}\ot a_{n+i,m-i,l})(y^j\ot y^l).\\
\end{array}$$
It follows that
$$\begin{array}{rl}
&\sum\limits_{i=0}^m\sum\limits_{j=0}^i
\left(\begin{array}{c}
i\\
j\\
\end{array}\right)_q\Delta(a_{n,m,i})(1\ot x^j)(y^j\ot y^{i-j})\\
\hspace{1.5cm}=&\sum\limits_{i=0}^m\sum\limits_{j=0}^i\sum\limits_{l=0}^{m-i}
\left(\begin{array}{c}
m\\
i\\
\end{array}\right)_q(a_{n,i,j}\ot a_{n+i,m-i,l})(y^j\ot y^l).\hspace{2cm}(2.3)\\
\end{array}$$
Similarly to the proof of Lemma \ref{primi}, by comparing the coefficients in $H_0\ot H_0$
of the basis element $1\ot y^m$ of the both sides of Eq.(2.3), one gets that
$\Delta(a_{n,m,m})=a_{n,0,0}\ot a_{n,m,m}=x^n\ot a_{n,m,m}$. Hence
$a_{n,m,m}=({\rm id}\ot\varepsilon)\Delta(a_{n,m,m})=\varepsilon(a_{n,m,m})x^n$.
Then from $a_{n,m,m}\neq 0$, one knows that $\varepsilon(a_{n,m,m})\neq 0$.
Let $\a_{n,m}=\varepsilon(a_{n,m,m})$. Then $\a_{n,m}\in k^{\times}$ and
$a_{n,m,m}=\a_{n,m}x^n$ for all $n\in\mathbb Z$ and $m\>1$.

Now suppose that $\a_{n,1}=1$ for all $n\in\mathbb{Z}$.
Note that Eq.(2.3) holds for all $n\in\mathbb Z$ and $m\>1$.
By comparing the coefficients in $H_0\ot H_0$
of the basis element $y^j\ot y^{m-j}$ of the both sides of Eq.(2.3), one gets that
$$\left(\begin{array}{c}
m\\
j\\
\end{array}\right)_q\Delta(a_{n,m,m})(1\ot x^j)=
\left(\begin{array}{c}
m\\
j\\
\end{array}\right)_q(a_{n,j,j}\ot a_{n+j,m-j,m-j}).$$
It follows that $\a_{n,m}=\a_{n,j}\a_{n+j, m-j}$ for all $n\in\mathbb Z$ and $0\<j\<m$,
where $\a_{n,0}=1$ since $a_{n,0,0}=x^n$. Thus, $\a_{n,m}=\a_{n,1}\a_{n+1, m-1}$.
Then by the induction on $m$, one can check that $\a_{n,m}=\a_{n,1}\a_{n+1,1}\cdots\a_{n+m-1,1}=1$.
This completes the proof.
\end{proof}

Let $Aut_c^{gr}(H)$ denote the graded automorphism group of the graded coalgebra $H$,
that is
$$Aut_c^{gr}(H)=\{\phi\in Aut_c(H)|\phi(H(n))\subseteq H(n), \forall\ n\geqslant 0\}.$$
Then obviously, $\Theta\subseteq Aut_c^{gr}(H)$. Let
$Aut_0^{gr}(H)=Aut_c^{gr}(H)\cap Aut_0(H)$. Then from Proposition \ref{Autosemi}, one gets the
following corollary.

\begin{corollary}\label{gr-semi}
$Aut_0^{gr}(H)$ is a normal subgroup of $Aut_c^{gr}(H)$. Moreover, $Aut_c^{gr}(H)$ is the
internal semidirect product of $Aut_0^{gr}(H)$ and $\Theta$.
\end{corollary}

It is enough to discuss the group structure of $Aut_0^{gr}(H)$ in order to discuss the group
structure of $Aut_c^{gr}(H)$. Let $\phi\in Aut_0^{gr}(H)$. Then by Lemmas \ref{auto01} and
\ref{firstitem}, there exists a family of nonzero scales
$\{\a_{n,m}\in k^{\ti}|n\in\mathbb{Z}, m\in\mathbb{N}\}$ such that
$\phi(x^ny^m)=\a_{n,m}x^ny^m$, where $\a_{n,0}=1$, $n\in\mathbb Z$, $m\in\mathbb N$.

\begin{lemma}\label{s-n-cond}
Let $\{\a_{n,m}\in k^{\ti}|n\in\mathbb{Z}, m\in\mathbb{N}\}$ be a family of nonzero scales.
Define a linear endomorphism $\phi: H\ra H$ by
$$\phi(x^ny^m)=\a_{n,m}x^ny^m,\ \ n\in\mathbb{Z},\ m\in\mathbb{N}.$$
Then $\phi$ is a coalgebra automorphism of $H$ if and only if
$\a_{n,m}=\a_{n,i}\a_{n+i,m-i}$ for all $n\in\mathbb{Z}$, $m\in\mathbb{N}$ and $0\leqslant i\leqslant m$.
\end{lemma}

\begin{proof}
Obviously, $\phi$ is a linear automorphism of $H$. Moreover, $\varepsilon\circ\phi=\varepsilon$
if and only if $\a_{n,0}=1$ for all $n\in\mathbb Z$.

Suppose that $\phi$ is a coalgebra map. Then by the proof of Lemma \ref{firstitem}, one knows that
$\a_{n,m}=\a_{n,i}\a_{n+i,m-i}$ for all $n\in\mathbb{Z}$ and $0\leqslant i\leqslant m$.
Conversely, suppose that $\a_{n,m}=\a_{n,i}\a_{n+i,m-i}$ for all
$n\in\mathbb{Z}$ and $0\leqslant i\leqslant m$.
By taking $i=m=0$, one gets that $\a_{n,0}=(\a_{n,0})^2$, and so
$\a_{n,0}=1$. Hence $\varepsilon\circ\phi=\varepsilon$.
Now let $n\in\mathbb Z$ and $m\in\mathbb N$. Then we have
$$\begin{array}{rcl}
(\phi\ot\phi)\Delta(x^ny^m)&=&
(\phi\ot\phi)(\sum\limits_{i=0}^m\left(\begin{array}{c}
m\\
i\\
\end{array}\right)_qx^ny^i\ot x^{n+i}y^{m-i})\\
&=&\sum\limits_{i=0}^m\left(\begin{array}{c}
m\\
i\\
\end{array}\right)_q\phi(x^ny^i)\ot\phi(x^{n+i}y^{m-i})\\
&=&\sum\limits_{i=0}^m\left(\begin{array}{c}
m\\
i\\
\end{array}\right)_q\a_{n,i}\a_{n+i,m-i}x^ny^i\ot x^{n+i}y^{m-i}.\\
&=&\a_{n,m}
\sum\limits_{i=0}^m\left(\begin{array}{c}
m\\
i\\
\end{array}\right)_qx^ny^i\ot x^{n+i}y^{m-i}\\
&=&\a_{n,m}\Delta(x^ny^m)=\Delta\phi(x^ny^m).\\
\end{array}$$
Thus, $\phi$ is a coalgebra map.
\end{proof}

Now let $\{\a_{n,m}\in k^{\ti}|n\in\mathbb{Z}, m\in\mathbb{N}\}$
be a family of nonzero scales satisfying $\a_{n,m}=\a_{n,i}\a_{n+i,m-i}$
for all $n\in\mathbb{Z}$ and $0\leqslant i\leqslant m$.
Then by the proofs of Lemmas \ref{firstitem} and \ref{s-n-cond}, one knows that
$\a_{n,0}=1$ and
$$\a_{n,m}=\a_{n,1}\a_{n+1,1}\cdots\a_{n+m-1,1}
=\prod\limits_{i=0}^{m-1}\a_{n+i,1},\ n\in\mathbb{Z}, m\geqslant 1.$$
Putting $\a_n=\a_{n,1}$, $n\in\mathbb{Z}$. Then $(\a_n)_{n\in\mathbb{Z}}\in(k^{\times})^{\mathbb Z}$
and $\a_{n,m}=\prod_{i=0}^{m-1}\a_{n+i}$,
where $n\in\mathbb{Z}$ and $m\geqslant 1$.

Conversely, given an element $\a=(\a_n)_{n\in\mathbb{Z}}\in(k^{\times})^{\mathbb Z}$.
Putting $\a_{n,0}=1$ and $\a_{n,m}=\prod_{i=0}^{m-1}\a_{n+i}$ for all
$n\in\mathbb{Z}$ and $m\geqslant 1$. Then one can easily check that
$\a_{n,m}=\a_{n,i}\a_{n+i,m-i}$ for all $n\in\mathbb{Z}$ and $0\leqslant i\leqslant m$.

Summarizing the above discussion, there is a 1-1 correspondence between
the groups $Aut_0^{gr}(H)$ and $(k^{\ti})^{\mathbb Z}$.
For an element $\a=(\a_n)_{n\in\mathbb{Z}}\in (k^{\ti})^{\mathbb Z}$, the corresponding
element $\phi_{\a}$ in $Aut_0^{gr}(H)$ is determined by
$$\hspace{1cm}\phi_{\a}(x^n)=x^n,\ \phi_{\a}(x^ny)=\a_nx^ny,\ \phi_{\a}(x^ny^m)
=(\prod\limits_{i=0}^{m-1}\a_{n+i})x^ny^m,\hspace{1cm}(2.4)$$
where $n\in\mathbb{Z}$ and $m\geqslant 2$.
Obviously, the above correspondence is a homomorphism of groups. Thus, we have proven
the following theorem.

\begin{theorem}\label{iso-gr-0}
$Aut_0^{gr}(H)$ and $(k^{\ti})^{\mathbb Z}$ are isomorphic groups. Precisely, the map
$$(k^{\ti})^{\mathbb Z}\ra Aut_0^{gr}(H),\ \a=(\a_n)_{n\in\mathbb{Z}}\mapsto \phi_{\a}$$
is a group isomorphism, where $\phi_{\a}$ is given by Eq.(2.4).
\end{theorem}

We have already known that the subgroup $\Theta$ of $Aut_c^{gr}(H)$ is isomorphic to the
additive group $\mathbb Z$. By Corollary \ref{gr-semi}, it follows that
$Aut_c^{gr}(H)$ is the internal semidirect product of $Aut_0^{gr}(H)$ and $\Theta$.
Therefore, there is an action of $\mathbb Z$ on $(k^{\ti})^{\mathbb Z}$ such that
the group $Aut_c^{gr}(H)$ is isomorphic to the corresponding semidirect product group
$(k^{\ti})^{\mathbb Z}\rtimes\mathbb{Z}$.

For $r\in\mathbb{Z}$ and $\a=(\a_n)_{n\in\mathbb{Z}}\in(k^{\ti})^{\mathbb Z}$,
define $r\cdot\a:=\a[-r]\in(k^{\ti})^{\mathbb Z}$, i.e.,
$(r\cdot\a)_n=\a_{n-r}$, $n\in\mathbb{Z}$. Obviously, the map
$(k^{\ti})^{\mathbb Z}\ra(k^{\ti})^{\mathbb Z}$, $\a\mapsto r\cdot\a$ is a group automorphism
of $(k^{\ti})^{\mathbb Z}$, and $(r+t)\cdot\a=r\cdot(t\cdot\a)$, $r, t\in\mathbb{Z}$,
$\a\in(k^{\ti})^{\mathbb Z}$. Hence one can form a semidirect group
$(k^{\ti})^{\mathbb Z}\rtimes\mathbb{Z}$ as follows:
$(k^{\ti})^{\mathbb Z}\rtimes\mathbb{Z}=\{(\a, r)|\a\in(k^{\ti})^{\mathbb Z}, r\in\mathbb{Z}\}$ as a set;
the multiplicative operation is defined by
$$(\a, r)(\b, t)=(\a(r\cdot\b), r+t),\ \a, \b\in(k^{\ti})^{\mathbb Z},\ r, t\in\mathbb{Z}.$$

\begin{theorem}\label{iso-gr}
$Aut_c^{gr}(H)$ is isomorphic to the semidirect product group $(k^{\ti})^{\mathbb Z}\rtimes\mathbb{Z}$.
Precisely, the map
$$\Psi: (k^{\ti})^{\mathbb Z}\rtimes\mathbb{Z}\ra Aut_c^{gr}(H),\ \Psi(\a, r)=\phi_{\a}\theta_r$$
is a group isomorphism, where $\theta_r$ and $\phi_{\a}$ are given as in Lemma \ref{Zpart} and
Theorem \ref{iso-gr-0}, respectively.
\end{theorem}

\begin{proof}
By Lemma \ref{Zpart}, Corollary \ref{gr-semi} and Theorem \ref{iso-gr-0}, one knows that $\Psi$ is a bijective map.
Let $\a, \b\in(k^{\ti})^{\mathbb Z}$ and $r, t\in\mathbb{Z}$. Then
$$\Psi(\a, r)\Psi(\b, t)=\phi_{\a}\theta_r\phi_{\b}\theta_t
=\phi_{\a}(\theta_r\phi_{\b}\theta_{-r})\theta_r\theta_t
=\phi_{\a}(\theta_r\phi_{\b}\theta_{-r})\theta_{r+t}.$$
It follows from Corollary \ref{gr-semi} that $\theta_r\phi_{\b}\theta_{-r}\in Aut_0^{gr}(H)$.
Now for any $n\in\mathbb{Z}$, we have
$$\begin{array}{rcl}
(\theta_r\phi_{\b}\theta_{-r})(x^ny)&=&(\theta_r\phi_{\b})(x^{n-r}y)=\theta_r(\b_{n-r}x^{n-r}y)\\
&=&\b_{n-r}x^ny=(r\cdot\b)_nx^ny=\phi_{r\cdot\b}(x^ny).\\
\end{array}$$
Again by Theorem \ref{iso-gr-0}, one knows that $\theta_r\phi_{\b}\theta_{-r}=\phi_{r\cdot\b}$.
It follows that
$$\begin{array}{rcl}
\Psi(\a, r)\Psi(\b, t)&=&\phi_{\a}\phi_{r\cdot\b}\theta_{r+t}=\phi_{\a(r\cdot\b)}\theta_{r+t}\\
&=&\Psi(\a(r\cdot\b), r+t)=\Psi((\a, r)(\b, t)).\\
\end{array}$$
Thus, $\Psi$ is a group homomorphism, and consequently $\Psi$ is a group isomorphism.
\end{proof}

\section{\bf The Coalgebra Automorphisms of $k_q[x, x^{-1}, y]$}\selabel{3}

In this section, we consider the coalgebra automorphism group of $k_q[x, x^{-1}, y]$.
We will use the notations in the last section. From Proposition \ref{Autosemi}, we only need
to consider the normal subgroup $Aut_0(H)$ of $Aut_c(H)$ in order to describe the structure of
$Aut_c(H)$.

For any $m\>1$, let
$$Aut_m(H)=\{\phi\in Aut_c(H)\mid\phi(h)=h\ \mbox{ for all }\ h\in\sum_{i=0}^mH(i)\}.$$
Then $Aut_m(H)$ is obviously a subgroup of $Aut_c(H)$.

\begin{lemma}\label{serial}
Each $Aut_m(H)$ is a normal subgroup of $Aut_c(H)$, where $m\>1$. Moreover,
$Aut_0(H)\supseteq Aut_1(H)\supseteq Aut_2(H)\supseteq\cdots$.
\end{lemma}

\begin{proof}
Let $m\>1$, $\phi\in Aut_m(H)$ and $\psi\in Aut_c(H)$. Then for any $h\in\sum_{i=0}^mH(i)$, by
Proposition \ref{Autosemi} and Lemma \ref{inclu}, we have
$\psi(h)\in\sum_{i=0}^mH(i)$. Hence $(\psi^{-1}\phi\psi)(h)=\psi^{-1}(\phi(\psi(h)))=\psi^{-1}(\psi(h))=h$
for any $h\in\sum_{i=0}^mH(i)$. This shows that $\psi^{-1}\phi\psi\in Aut_m(H)$, and consequently
$Aut_m(H)$ is a normal subgroup of $Aut_c(H)$. Obviously,
$Aut_0(H)\supseteq Aut_1(H)\supseteq Aut_2(H)\supseteq\cdots$.
\end{proof}

Let $\phi\in Aut_0(H)$. Then by Lemma \ref{auto01}, there are $(\a_n)_{n\in\mathbb Z}\in(k^{\times})^{\mathbb Z}$
and $(\b_n)_{n\in\mathbb Z}\in k^{\mathbb Z}$ such that $\phi(x^ny)=\a_nx^ny+\b_n(x^{n+1}-x^n)$ for all $n\in\mathbb Z$.
Let
$$Aut_*(H)=\{\phi\in Aut_0(H)\mid\phi(x^ny)=x^ny+\b_n(x^{n+1}-x^n), \b_n\in k,
n\in{\mathbb Z}\}.$$
Then $Aut_*(H)$ is obviously a subgroup of $Aut_c(H)$, and $Aut_1(H)\subseteq Aut_*(H)$.

\begin{lemma}\label{Semidirect}
$Aut_*(H)$ is a normal subgroup of $Aut_c(H)$. Moreover, $Aut_c(H)$ is the internal
semidirect product of $Aut_*(H)$ and $Aut^{gr}_c(H)$. Consequently,
$Aut_0(H)$ is the internal
semidirect product of $Aut_*(H)$ and $Aut^{gr}_0(H)$.
\end{lemma}

\begin{proof}
Let $\phi\in Aut_*(H)$ and $\psi\in Aut_c(H)$. Then by Lemma \ref{auto01},
there is an $r\in\mathbb Z$ such that
$$\phi(x^ny)=x^ny+\b_n(x^{n+1}-x^n),\ \psi(x^ny)=\a_nx^{n+r}y+\g_n(x^{n+r+1}-x^{n+r}),$$
where $n\in\mathbb Z$, $\a_n\in k^{\times}$ and $\b_n, \g_n\in k$. In this case,
$\psi^{-1}(x^{n+r}y)=\a_n^{-1}x^ny-\a_n^{-1}\g_n(x^{n+1}-x^n)$. Hence
$$\begin{array}{rcl}
(\psi^{-1}\phi\psi)(x^ny)&=&(\psi^{-1}\phi)(\a_nx^{n+r}y+\g_n(x^{n+r+1}-x^{n+r}))\\
&=&\psi^{-1}(\a_nx^{n+r}y+(\a_n\b_{n+r}+\g_n)(x^{n+r+1}-x^{n+r}))\\
&=&x^ny+\a_n\b_{n+r}(x^{n+1}-x^n))\\
\end{array}$$
for all $n\in\mathbb Z$. This shows that $Aut_*(H)$ is a normal subgroup of $Aut_c(H)$.

Let $\a=(\a_n)_{n\in\mathbb Z}\in(k^{\times})^{\mathbb Z}$ with $\a_n$ given above. Then
$\phi_{\a}\in Aut^{gr}_0(H)\subseteq Aut^{gr}_c(H)$ as stated in \seref{2}.
In this case, $(\phi_{\a}^{-1}\theta_{-r}\psi)(x^ny)=x^ny+\g_n(x^{n+1}-x^n)$,
and hence $\phi_{\a}^{-1}\theta_{-r}\psi\in Aut_*(H)$. Thus,
$\psi=(\theta_r\phi_{\a})(\phi_{\a}^{-1}\theta_{-r}\psi)\in Aut^{gr}_c(H)Aut_*(H)$,
and so $Aut_c(H)=Aut^{gr}_c(H)Aut_*(H)$. Similarly, one can show that
$Aut_c(H)=Aut_*(H)Aut^{gr}_c(H)$. From Theorem \ref{iso-gr-0}, it is easy to see that
$Aut_*(H)\cap Aut_c^{gr}(H)=\{{\rm id}_H\}$. It follows that $Aut_c(H)$ is the internal
semidirect product of $Aut_*(H)$ and $Aut^{gr}_c(H)$.  Consequently,
$Aut_0(H)$ is the internal
semidirect product of $Aut_*(H)$ and $Aut^{gr}_0(H)$.
\end{proof}

\begin{corollary}\label{Iso0}
There are group isomorphisms
$$Aut_c(H)/Aut_*(H)\cong Aut_c^{gr}(H)\cong (k^{\times})^{\mathbb Z}\rtimes\mathbb{Z}$$
and
$$Aut_0(H)/Aut_*(H)\cong Aut_0^{gr}(H)\cong (k^{\times})^{\mathbb Z}.$$
\end{corollary}

\begin{proof}
It follows from Lemma \ref{Semidirect} and Theorems \ref{iso-gr-0} and \ref{iso-gr}.
\end{proof}

%

For any $\b=(\b_n)_{n\in\mathbb{Z}}\in k^{\mathbb{Z}}$ (or $(k^{\times})^{\mathbb{Z}}$),
define $\b_{n,m}\in k$ (or $k^{\times}$) for all $n\in\mathbb Z$ and $m\>0$
by $\b_{n,0}=1$ and
$\b_{n,m}=\b_n\b_{n+1}\cdots\b_{n+m-1}=\prod_{i=0}^{m-1}\b_{n+i}$ for $m\>1$.
For any integers $0<m\<n$, define $(n,m)_q$ by
$(n,m)_q=(n)_q(n-1)_q\cdots(n-m+1)_q=\prod_{i=0}^{m-1}(n-i)_q$.

Note that $k^{\mathbb Z}$ is an additive group as stated in Section 1.
Let $\phi\in Aut_*(H)$. Then there is a unique element $\b=(\b_n)_{n\in\mathbb{Z}}\in k^{\mathbb{Z}}$
such that $\phi(x^ny)=x^ny+\b_n(x^{n+1}-x^n)$ for all $n\in\mathbb Z$. Hence one can
define a map $f_1: Aut_*(H)\ra k^{\mathbb Z}$ by $f_1(\phi)=\b=(\b_n)_{n\in\mathbb{Z}}$.

\begin{proposition}\label{Iso1}
$f_1$ is a group epimorphism with ${\rm Ker}(f_1)=Aut_1(H)$. That is, there is an exact sequence of groups
$$1\ra Aut_1(H)\hookrightarrow Aut_*(H)\xrightarrow{f_1}k^{\mathbb Z}\ra 0.$$
\end{proposition}

\begin{proof}
Let $\phi, \psi\in Aut_*(H)$, and assume that $f_1(\phi)=\b=(\b_n)_{n\in\mathbb{Z}}$ and
$f_1(\psi)=\g=(\g_n)_{n\in\mathbb{Z}}$ in $k^{\mathbb Z}$. Then
$\phi(x^n)=x^n$, $\phi(x^ny)=x^ny+\b_n(x^{n+1}-x^n)$, $\psi(x^n)=x^n$ and $\psi(x^ny)=x^ny+\g_n(x^{n+1}-x^n)$,
$n\in\mathbb Z$. Hence $(\phi\psi)(x^n)=x^n$ and $(\phi\psi)(x^ny)=\phi(x^ny+\g_n(x^{n+1}-x^n))
=x^ny+(\b_n+\g_n)(x^{n+1}-x^n)$
for all $n\in\mathbb Z$, and so $f_1(\phi\psi)=\b+\g=f_1(\phi)+f_1(\psi)$.
It follows that $f_1$ is a group homomorphism. Obviously, Ker$(f_1)=Aut_1(H)$. It is left to show that
$f_1$ is surjective.

Let $\b=(\b_n)_{n\in\mathbb{Z}}\in k^{\mathbb Z}$.
For any $n\in\mathbb Z$ and $0\<l\<m$, define $a_{n,m,l}\in H_0$ by $a_{n,m,m}=x^n$ for $m\>0$,
and
$$a_{n,m,l}=(m,m-l)_q
(\b_{n,m-l}x^{n+m-l}-\b_{n,m-l-1}\b_{n+m-1}x^{n+m-l-1})$$
for $0\<l<m$. Note that $a_{n,m,0}=(m)!_q\b_{n, m}(x^{n+m}-x^{n+m-1})$ if $m>0$.
Now define a linear map $\phi^{(1)}_{\b}: H\ra H$ by
$$\phi^{(1)}_{\b}(x^ny^m)=\sum\limits_{l=0}^ma_{n,m,l}y^l,\
n\in\mathbb{Z},\ m\>0.$$
It is easy to see that $\phi^{(1)}_{\b}$ is a bijection and
$\e\phi^{(1)}_{\b}(x^ny^m)=\e(x^ny^m)$ for all $n\in\mathbb Z$ and $m\>0$.
Now we are going to show that $\Delta\phi^{(1)}_{\b}(x^ny^m)=(\phi^{(1)}_{\b}\ot\phi_{\b}^{(1)})\Delta(x^ny^m)$
for all $n\in\mathbb Z$ and $m\>0$.
From the proof of Lemma \ref{firstitem}, we have
$$\Delta\phi^{(1)}_{\b}(x^ny^m)=\sum\limits_{l=0}^m\sum\limits_{j=0}^l
\left(\begin{array}{c}
l\\
j\\
\end{array}\right)_q\Delta(a_{n,m,l})(1\ot x^j)(y^j\ot y^{l-j})$$
and
$$
(\phi^{(1)}_{\b}\ot\phi_{\b}^{(1)})\Delta(x^ny^m)=\sum\limits_{i=0}^m\sum\limits_{j=0}^i\sum\limits_{t=0}^{m-i}
\left(\begin{array}{c}
m\\
i\\
\end{array}\right)_q(a_{n,i,j}\ot a_{n+i,m-i,t})(y^j\ot y^t).$$
For any $0\<j\<l\<m$, the coefficients of $y^j\ot y^{l-j}$ in $\Delta\phi^{(1)}_{\b}(x^ny^m)$
and $(\phi^{(1)}_{\b}\ot\phi_{\b}^{(1)})\Delta(x^ny^m)$ are
$$\left(\begin{array}{c}
l\\
j\\
\end{array}\right)_q\Delta(a_{n,m,l})(1\ot x^j) \mbox{ and }
\sum\limits_{i=j}^{m-l+j}
\left(\begin{array}{c}
m\\
i\\
\end{array}\right)_qa_{n,i,j}\ot a_{n+i,m-i,l-j}$$
in $H_0\ot H_0$, respectively.
Thus, we only need to show that
$$\hspace{1.2cm}\left(\begin{array}{c}
l\\
j\\
\end{array}\right)_q\Delta(a_{n,m,l})(1\ot x^j)=
\sum\limits_{i=j}^{m-l+j}
\left(\begin{array}{c}
m\\
i\\
\end{array}\right)_q a_{n,i,j}\ot a_{n+i,m-i,l-j}\ \ \ \ \ \ (3.1)$$
for all $n\in\mathbb Z$ and $0\<j\<l\<m$.

If $l=m$, then
$$\left(\begin{array}{c}
l\\
j\\
\end{array}\right)_q\Delta(a_{n,m,l})(1\ot x^j)=\left(\begin{array}{c}
m\\
j\\
\end{array}\right)_q\Delta(a_{n,m,m})(1\ot x^j)=\left(\begin{array}{c}
m\\
j\\
\end{array}\right)_qx^n\ot x^{n+j}$$
and
$$\begin{array}{rcl}
\sum\limits_{i=j}^{m-l+j}
\left(\begin{array}{c}
m\\
i\\
\end{array}\right)_qa_{n,i,j}\ot a_{n+i,m-i,l-j}&=&
\left(\begin{array}{c}
m\\
j\\
\end{array}\right)_q a_{n,j,j}\ot a_{n+j,m-j,m-j}\\
&=&
\left(\begin{array}{c}
m\\
j\\
\end{array}\right)_qx^n\ot x^{n+j}.\\
\end{array}$$
Hence Eq.(3.1) holds in this case.

If $0\<j\<l<m$, let $\eta=\left(\begin{array}{c}
l\\
j\\
\end{array}\right)_q(m,m-l)_q$. Then
$$\begin{array}{cl}
&\left(\begin{array}{c}
l\\
j\\
\end{array}\right)_q\Delta(a_{n,m,l})(1\ot x^j)\\
=&\eta(\b_{n,m-l}x^{n+m-l}\ot x^{n+m-l+j}-\b_{n,m-l-1}\b_{n+m-1}x^{n+m-l-1}\ot x^{n+m-l-1+j})\\
\end{array}$$
and
$$\begin{array}{rl}
&\sum\limits_{i=j}^{m-l+j}
\left(\begin{array}{c}
m\\
i\\
\end{array}\right)_q a_{n,i,j}\ot a_{n+i,m-i,l-j}\\
=&\left(\begin{array}{c}
m\\
j\\
\end{array}\right)_q a_{n,j,j}\ot a_{n+j,m-j,l-j}
+\sum\limits_{j<i<m-l+j}
\left(\begin{array}{c}
m\\
i\\
\end{array}\right)_q a_{n,i,j}\ot a_{n+i,m-i,l-j}\\
&+\left(\begin{array}{c}
m\\
m-l+j\\
\end{array}\right)_q a_{n,m-l+j,j}\ot a_{n+m-l+j,l-j,l-j}\\
=&\left(\begin{array}{c}
m\\
j\\
\end{array}\right)_q(m-j,m-l)_q\\
&\times x^n\ot (\b_{n+j,m-l}x^{n+m-l+j}-\b_{n+j,m-l-1}\b_{n+m-1}x^{n+m-l+j-1})\\
&+\sum\limits_{j<i<m-l+j}
\left(\begin{array}{c}
m\\
i\\
\end{array}\right)_q (i,i-j)_q (m-i,m-i-l+j)_q\\
&\times(\b_{n,i-j}x^{n+i-j}-\b_{n,i-j-1}\b_{n+i-1}x^{n+i-j-1})\\
&\ot (\b_{n+i,m-i-l+j}x^{n+m-l+j}-\b_{n+i,m-i-l+j-1}\b_{n+m-1}x^{n+m-l+j-1})\\
&+\left(\begin{array}{c}
m\\
m-l+j\\
\end{array}\right)_q (m-l+j,m-l)_q\\
&\times(\b_{n,m-l}x^{n+m-l}-\b_{n,m-l-1}\b_{n+m-l+j-1}x^{n+m-l-1}) \ot x^{n+m-l+j}\\
=&\eta x^n\ot
(\b_{n+j,m-l}x^{n+m-l+j}-\b_{n+j,m-l-1}\b_{n+m-1}x^{n+m-l+j-1})\\
&+\sum\limits_{j<i<m-l+j}
\eta
(\b_{n,i-j}x^{n+i-j}-\b_{n,i-j-1}\b_{n+i-1}x^{n+i-j-1})\\
&\ot (\b_{n+i,m-i-l+j}x^{n+m-l+j}-\b_{n+i,m-i-l+j-1}\b_{n+m-1}x^{n+m-l+j-1})\\
&+\eta(\b_{n,m-l}x^{n+m-l}-\b_{n,m-l-1}\b_{n+m-l+j-1}x^{n+m-l-1})
\ot x^{n+m-l+j}\\
=&\eta(\b_{n+j,m-l}x^n\ot x^{n+m-l+j}-\b_{n+j,m-l-1}\b_{n+m-1}x^n\ot x^{n+m-l+j-1})\\
&+\eta\sum\limits_{j<i<m-l+j}
\b_{n,i-j}\b_{n+i,m-i-l+j}x^{n+i-j}\ot x^{n+m-l+j}\\
&-\eta\sum\limits_{j<i<m-l+j}
\b_{n,i-j}\b_{n+i,m-i-l+j-1}\b_{n+m-1}x^{n+i-j}\ot x^{n+m-l+j-1}\\
\end{array}$$
$$\begin{array}{rl}
&-\eta\sum\limits_{j<i<m-l+j}
\b_{n,i-j-1}\b_{n+i-1,m-i-l+j+1}x^{n+i-j-1}\ot x^{n+m-l+j}\\
&+\eta\sum\limits_{j<i<m-l+j}
\b_{n,i-j-1}\b_{n+i-1,m-i-l+j}\b_{n+m-1}x^{n+i-j-1}\ot x^{n+m-l+j-1})\\
&+\eta(\b_{n,m-l}x^{n+m-l}\ot x^{n+m-l+j}-\b_{n,m-l-1}\b_{n+m-l+j-1}x^{n+m-l-1}
\ot x^{n+m-l+j})\\
=&\eta(\b_{n,m-l}x^{n+m-l}\ot x^{n+m-l+j}-\b_{n,m-l-1}\b_{n+m-1}x^{n+m-l-1}\ot x^{n+m-l-1+j}).\\
\end{array}$$
This shows that Eq.(3.1) also holds in this case.
Thus, we have proved that $\phi^{(1)}_{\b}$ is a coalgebra automorphism of $H$.
Since $\phi^{(1)}_{\b}(x^n)=a_{n,0,0}=x^n$ and $\phi^{(1)}_{\b}(x^ny)=a_{n,1,1}y+a_{n,1,0}
=x^ny+\b_n(x^{n+1}-x^n)$ for all $n\in\mathbb Z$, $\phi^{(1)}_{\b}\in Aut_*(H)$.
Moreover, $f_1(\phi^{(1)}_{\b})=\b$. It follows that $f_1$ is surjective.
\end{proof}

\begin{lemma}\label{0item}
Let $s\>2$ and $\phi\in Aut_{s-1}(H)$. Then there exists a unique
$\b=(\b_n)_{n\in\mathbb{Z}}\in k^{\mathbb Z}$ such that
$\phi(x^ny^{s})=x^ny^{s}+\b_n(x^{n+s}-x^n)$ for all $n\in\mathbb{Z}.$
\end{lemma}
\begin{proof}
For any $n\in\mathbb Z$ and $m\>0$, it follows from Lemma \ref{firstitem} that
$\phi(x^ny^m)=\sum_{l=0}^ma_{n,m,l}y^l$ for some $a_{n,m,l}\in H_0$
since $\phi\in Aut_{s-1}(H)\subseteq Aut_*(H)\subseteq Aut_0(H)$. Moreover, we have that
$a_{n,m,m}=x^n$ for all $m\>0$, and $a_{n, m, l}=0$ for all $0\<l<m\<s-1$.
Since $\phi$ is a coalgebra map, $\Delta\phi(x^ny^{s})=(\phi\ot\phi)\Delta(x^ny^{s})$
for all $n\in\mathbb Z$. Then by the proof of Proposition \ref{Iso1}, we have
$$\hspace{1.2cm}\left(\begin{array}{c}
l\\
j\\
\end{array}\right)_q\Delta(a_{n,s,l})(1\ot x^j)=
\sum\limits_{i=j}^{s-l+j}
\left(\begin{array}{c}
s\\
i\\
\end{array}\right)_q a_{n,i,j}\ot a_{n+i,s-i,l-j}\ \ \ \ \ \ \ \ \ (3.2)$$
for all $n\in\mathbb Z$ and $0\<j\<l\<s$. Putting $j=1$ in Eq.(3.2),
we have
$$\hspace{1.5cm}(l)_q\Delta(a_{n,s,l})(1\ot x)=
\sum\limits_{i=1}^{s-l+1}
\left(\begin{array}{c}
s\\
i\\
\end{array}\right)_q a_{n,i,1}\ot a_{n+i,s-i,l-1}.\quad\quad\quad\quad\ (3.3)$$
If $2\<l<s$, then $1<s-l+1\<s-1$ and $a_{n,i,1}=0$ for all $2\<i\<s-l+1$.
Thus, by Eq.(3.2) we have
$$(l)_q\Delta(a_{n,s,l})(1\ot x)=\left(\begin{array}{c}
s\\
1\\
\end{array}\right)_q a_{n,1,1}\ot a_{n+1,s-1,l-1}=0$$
since $a_{n+1,s-1,l-1}=0$ by $s-1>l-1$. It follows that $a_{n,s,l}=0$ for all
$2\<l<s$.
If $l=1$, then by Eq.(3.3) we have
$$\begin{array}{rcl}
\Delta(a_{n,s,1})(1\ot x)&=&
\sum\limits_{i=1}^{s}
\left(\begin{array}{c}
s\\
i\\
\end{array}\right)_q a_{n,i,1}\ot a_{n+i,s-i,0}\\
&=&a_{n,s,1}\ot a_{n+s,0,0}\\
&=&a_{n,s,1}\ot x^{n+s}.\\
\end{array}$$
Applying $\varepsilon\ot{\rm id}$ on the both sides of the above equation, one gets that
$a_{n,s,1}=\varepsilon(a_{n,s,1})x^{n+s-1}$. Now putting $j=0$ and $l=1$ in Eq.(3.2), we have
$$\Delta(a_{n,s,1})=
\sum\limits_{i=0}^{s-1}
\left(\begin{array}{c}
s\\
i\\
\end{array}\right)_q a_{n,i,0}\ot a_{n+i,s-i,1}=a_{n,0,0}\ot a_{n,s,1}=x^n\ot a_{n,s,1}.$$
Applying ${\rm id}\ot\varepsilon$ on the both sides of the above equation, one gets that
$a_{n,s,1}=\varepsilon(a_{n,s,1})x^{n}$. Then from $a_{n,s,1}=\varepsilon(a_{n,s,1})x^{n+s-1}$
and $a_{n,s,1}=\varepsilon(a_{n,s,1})x^{n}$, one gets that $a_{n,s,1}=0$.
Finally, putting $j=l=0$ in Eq.(3.2), we have
$$\begin{array}{rcl}
\Delta(a_{n,s,0})&=&
\sum\limits_{i=0}^{s}
\left(\begin{array}{c}
s\\
i\\
\end{array}\right)_q a_{n,i,0}\ot a_{n+i,s-i,0}\\
&=&a_{n,0,0}\ot a_{n,s,0}+a_{n,s,0}\ot a_{n+s,0,0}\\
&=&x^n\ot a_{n,s,0}+a_{n,s,0}\ot x^{n+s}.\\
\end{array}$$
It follows from Lemma \ref{primi} that $a_{n,s,0}=\b_n(x^{n+s}-x^n)$ for some $\b_n\in k$,
and so $\phi(x^ny^{s})=x^ny^{s}+\b_n(x^{n+s}-x^n)$. Obviously,
the element $\b=(\b_n)_{n\in\mathbb{Z}}\in k^{\mathbb Z}$ is uniquely determined by $\phi$.
\end{proof}

Let $s\>2$. By Lemma \ref{0item}, one can define a map $f_s: Aut_{s-1}(H)\ra k^{\mathbb Z}$,
$f_s(\phi)=\b=(\b_n)_{n\in\mathbb{Z}}$ by $\phi(x^ny^{s})=x^ny^{s}+\b_n(x^{n+s}-x^n)$ for all $n\in\mathbb Z$.
Then we have the following lemma.

\begin{lemma}\label{f_s-homo}
Let $s\>2$. Then the map $f_s: Aut_{s-1}(H)\ra k^{\mathbb Z}$ defined above is a group homomorphism
with ${\rm Ker}(f_s)=Aut_{s}(H)$.
\end{lemma}

\begin{proof}
Let $\phi, \psi\in Aut_{s-1}(H)$, and assume that $f_s(\phi)=\b=(\b_n)_{n\in\mathbb Z}$ and
$f_s(\psi)=\g=(\g_n)_{n\in\mathbb Z}$ in $k^{\mathbb Z}$. Then $\phi(x^ny^m)=\psi(x^ny^m)=x^ny^m$
for all $n\in\mathbb Z$ and $0\<m\<s-1$, and
$$\phi(x^ny^{s})=x^ny^{s}+\b_n(x^{n+s}-x^n),\ \ \psi(x^ny^{s})=x^ny^{s}+\g_n(x^{n+s}-x^n)$$
for all $n\in\mathbb Z$. Hence $(\phi\psi)(x^ny^m)=x^ny^m$
for all $n\in\mathbb Z$ and $0\<m\<s-1$, and
$$\begin{array}{rcl}
(\phi\psi)(x^ny^{s})&=&\phi(x^ny^{s}+\g_n(x^{n+s}-x^n))\\
&=&\phi(x^ny^{s})+\g_n(x^{n+s}-x^n)\\
&=&x^ny^{s}+(\b_n+\g_n)(x^{n+s}-x^n)\\
\end{array}$$
for all $n\in\mathbb Z$. Thus, $f_s(\phi\psi)=\b+\g=f_s(\phi)+f_s(\psi)$.
This shows that $f_s$ is a group homomorphism. Obviously, ${\rm Ker}(f_s)=Aut_{s}(H)$.
\end{proof}

For any integers $1\<t\<m$ and $1\<l\<\frac{m}{t}$, let
$$\left(
    \begin{array}{c}
      m \\
      t \\
    \end{array}
  \right)_{q,l}=\left(
    \begin{array}{c}
      m \\
      t \\
    \end{array}
  \right)_q\left(
    \begin{array}{c}
      m-t \\
      t \\
    \end{array}
  \right)_q\cdots\left(
    \begin{array}{c}
      m-(l-1)t \\
      t \\
    \end{array}
  \right)_q=\prod\limits_{i=0}^{l-1}\left(
    \begin{array}{c}
      m-it \\
      t \\
    \end{array}
  \right)_q.$$
Note that $\left(
    \begin{array}{c}
      m \\
      t \\
    \end{array}
  \right)_{q,1}=\left(
    \begin{array}{c}
      m \\
      t \\
    \end{array}
  \right)_q$ and $\left(
    \begin{array}{c}
      m \\
      1 \\
    \end{array}
  \right)_{q,l}=(m,l)_q$.

Let $\b=(\b_n)_{n\in\mathbb Z}\in k^{\mathbb Z}$. For any integers $n\in\mathbb{Z}$, $t\>1$ and $m\>0$,
define $\b_{n,t;m}\in k$ by $\b_{n,t;0}=1$ and
$$\b_{n,t;m}=\b_n\b_{n+t}\cdots\b_{n+(m-1)t}=\prod\limits_{i=0}^{m-1}\b_{n+it}$$
for $m>0$. Note that $\b_{n,t;1}=\b_n$ and $\b_{n,1;m}=\b_{n,m}$.

For any real number $r$, let $[r]$ denote the integral part of $r$, i.e., $[r]$ is an integer
such that $0\<r-[r]<1$.

\begin{proposition}\label{Isos}
Let $s\>2$. Then $f_s: Aut_{s-1}(H)\ra k^{\mathbb Z}$ is a group epimorphism. Consequently,
there is an exact sequence of groups
$$1\ra Aut_s(H)\hookrightarrow Aut_{s-1}(H)\xrightarrow{f_s}k^{\mathbb Z}\ra 0.$$
\end{proposition}

\begin{proof}
By Lemma \ref{f_s-homo}, it is enough to show that $f_s$ is surjective.
The proof is similar to Proposition \ref{Iso1}.

Let $\b=(\b_n)_{n\in\mathbb{Z}}\in k^{\mathbb Z}$.
For any $n\in\mathbb Z$ and $0\<l\<m$, define $a_{n,m,l}\in H_0$ by $a_{n,m,m}=x^n$ for $m\>0$,
$a_{n,m,l}=0$ for $0\<l<m$ with $s\nmid m-l$, and
$$a_{n,m,l}=\left(\begin{array}{c}
m\\
s\\
\end{array}\right)_{q,\frac{m-l}{s}}
(\b_{n,s;\frac{m-l}{s}}x^{n+m-l}-\b_{n,s;\frac{m-l}{s}-1}\b_{n+m-s}x^{n+m-l-s})$$
for $0\<l<m$ with $s|m-l$. Note that $a_{n,m,0}=\left(\begin{array}{c}
m\\
s\\
\end{array}\right)_{q,\frac{m}{s}}
\b_{n,s;\frac{m}{s}}(x^{n+m}-x^{n+m-s})$
if $m>0$ with $s|m$. Hence $\varepsilon(a_{n,m,0})=0$ for all $n\in\mathbb Z$ and
$m>0$. Now define a linear map $\phi^{(s)}_{\b}: H\ra H$ by
$$\phi^{(s)}_{\b}(x^ny^m)=\sum\limits_{l=0}^ma_{n,m,l}y^l,\
n\in\mathbb{Z},\ m\>0.$$
It is easy to see that $\phi^{(s)}_{\b}$ is bijective and
$\e\phi^{(s)}_{\b}(x^ny^m)=\e(x^ny^m)$ for all $n\in\mathbb Z$ and $m\>0$.
In order to show that $\phi^{(s)}_{\b}$ is a coalgebra map, by the proof of Proposition \ref{Iso1},
we only need to show that
Eq.(3.1) holds for all $n\in\mathbb{Z}$  and $0\<j\<l\<m$.

Case1: $l=m$. In this case, an argument same as the proof of Proposition \ref{Iso1} shows that
Eq.(3.1) holds for all $n\in\mathbb{Z}$ and $0\<j\<m$.

Case 2: $0\<j\<l<m$ with $s\nmid m-l$. In this case, $(m-i)-(l-j)=(m-l)-(i-j)$, and so
$s\nmid i-j$ or $s\nmid (m-i)-(l-j)$ for all $j\<i\<m-l+j$.
Hence the both sides of Eq.(3.1) are equal to zero.

Case 3: $0\<j\<l<m$ with $s\mid m-l$. Let $t=\frac{m-l}{s}$ and
$\eta=\left(\begin{array}{c}
l\\
j\\
\end{array}\right)_q\left(\begin{array}{c}
m\\
s\\
\end{array}\right)_{q,t}$ in this case. Then
$$\begin{array}{cl}
&\left(\begin{array}{c}
l\\
j\\
\end{array}\right)_q\Delta(a_{n,m,l})(1\ot x^j)\\
=&\eta
(\b_{n,s;t}x^{n+m-l}\ot x^{n+m-l+j}
-\b_{n,s;t-1}\b_{n+m-s}x^{n+m-l-s}\ot x^{n+m-l-s+j})\\
\end{array}$$
and
$$\begin{array}{cl}
&\sum\limits_{i=j}^{m-l+j}
\left(\begin{array}{c}
m\\
i\\
\end{array}\right)_q a_{n,i,j}\ot a_{n+i,m-i,l-j}\\
=&\sum\limits_{i=0}^t
\left(\begin{array}{c}
m\\
j+is\\
\end{array}\right)_q a_{n,j+is,j}\ot a_{n+j+is,m-j-is,l-j}\\
\end{array}$$
$$\begin{array}{rl}
=&\left(\begin{array}{c}
m\\
j\\
\end{array}\right)_qa_{n,j,j}\ot a_{n+j,m-j,l-j}\\
&+\sum\limits_{0<i<t}
\left(\begin{array}{c}
m\\
j+is\\
\end{array}\right)_q a_{n,j+is,j}\ot a_{n+j+is,m-j-is,l-j}\hspace{2cm}\\
&+\left(\begin{array}{c}
m\\
m-l+j\\
\end{array}\right)_qa_{n,m-l+j,j}\ot a_{n+m-l+j,l-j,l-j}\\
=&\left(\begin{array}{c}
m\\
j\\
\end{array}\right)_q\left(\begin{array}{c}
m-j\\
s\\
\end{array}\right)_{q,t}\\
&\times x^n\ot(\b_{n+j,s;t}x^{n+m-l+j}-\b_{n+j,s;t-1}\b_{n+m-s}x^{n+m-l+j-s})\\
&+\sum\limits_{0<i<t}
\left(\begin{array}{c}
m\\
j+is\\
\end{array}\right)_q\left(\begin{array}{c}
j+is\\
s\\
\end{array}\right)_{q,i}\left(\begin{array}{c}
m-j-is\\
s\\
\end{array}\right)_{q,t-i}\\
&\times(\b_{n,s;i}x^{n+is}-\b_{n,s;i-1}\b_{n+j+is-s}x^{n+is-s})\\
&\ot
(\b_{n+j+is,s;t-i}x^{n+m-l+j}-\b_{n+j+is,s;t-i-1}\b_{n+m-s}x^{n+m-l+j-s})\\
&+\left(\begin{array}{c}
m\\
m-l+j\\
\end{array}\right)_q\left(\begin{array}{c}
m-l+j\\
s\\
\end{array}\right)_{q,t}\\
&\times(\b_{n,s;t}x^{n+m-l}-\b_{n,s;t-1}\b_{n+m-l+j-s}x^{n+m-l-s})\ot x^{n+m-l+j}\\
=&\eta x^n\ot
(\b_{n+j,s;t}x^{n+m-l+j}-\b_{n+j,s;t-1}\b_{n+m-s}x^{n+m-l+j-s})\\
&+\eta\sum\limits_{0<i<t}
(\b_{n,s;i}x^{n+is}-\b_{n,s;i-1}\b_{n+j+is-s}x^{n+is-s})\\
&\ot (\b_{n+j+is,s;t-i}x^{n+m-l+j}-\b_{n+j+is,s;t-i-1}\b_{n+m-s}x^{n+m-l+j-s})\\
&+\eta
(\b_{n,s;t}x^{n+m-l}-\b_{n,s;t-1}\b_{n+m-l+j-s}x^{n+m-l-s})\ot x^{n+m-l+j}\\
=&\eta
(\b_{n,s;t}x^{n+m-l}\ot x^{n+m-l+j}
-\b_{n,s;t-1}\b_{n+m-s}x^{n+m-l-s}\ot x^{n+m-l-s+j}).\\
\end{array}$$
Hence Eq.(3.1) holds in this case.

Thus, we have proved that $\phi^{(s)}_{\b}$ is a coalgebra automorphism of $H$. Obviously,
$\phi^{(s)}_{\b}\in Aut_{s-1}(H)$ and $\phi^{(s)}_{\b}(x^ny^s)=x^ny^s+\b_n(x^{n+s}-x^n)$ for all
$n\in\mathbb Z$. It follows that $f_s(\phi^{(s)}_{\b})=\b$, and so $f_s$ is surjective.
This completes the proof.
\end{proof}

\begin{remark}\label{formula}
Let $s\>2$ and $\b=(\b_n)_{n\in\mathbb Z}\in k^{\mathbb Z}$. For any $n\in\mathbb{Z}$ and $m\>0$,
$\phi^{(s)}_{\b}(x^ny^m)=x^ny^m$ for $0\<m<s$ and
$$\phi^{(s)}_{\b}(x^ny^m)=x^ny^m+\sum\limits_{1\<i\<[\frac{m}{s}]}
\left(\begin{array}{c}
m\\
s\\
\end{array}\right)_{q,i}
(\b_{n,s;i}x^{n+is}-\b_{n,s;i-1}\b_{n+m-s}x^{n+is-s})y^{m-is}$$
for $m\>s$. In particular,
$\phi^{(s)}_{\b}(x^ny^s)=x^ny^s+
\b_n(x^{n+s}-x^n)$ for all $n\in\mathbb{Z}$.
\end{remark}

From Lemma \ref{Semidirect} and Corollary \ref{Iso0}, in order to describe the structure of the group $Aut_c(H)$, we only
need to describe the structure of the normal subgroup $Aut_*(H)$. We have already known that
$Aut_s(H)$ is a normal subgroup of $Aut_*(H)$ for all $s\>1$, and
$$Aut_*(H)\supseteq Aut_1(H)\supseteq Aut_2(H)\supseteq\cdots.$$
Obviously, $\bigcap_{s\>1}Aut_s(H)=\{\rm id_H\}$.

For any $i\>1$, let $\pi_i: Aut_*(H)\ra Aut_*(H)/Aut_i(H)$
be the canonical group epimorphism. If $1\<i\<j$, then
there is a canonical group epimorphism $\pi^j_i: Aut_*(H)/Aut_j(H)\ra Aut_*(H)/Aut_i(H)$
with  ${\rm Ker}(\pi^j_i)=Aut_i(H)/Aut_j(H)$. $\pi^j_i$ is the unique group morphism such that $\pi^j_i\pi_j=\pi_i$.
Obviously, $\pi^j_j={\rm id}$ and $\pi^j_i\pi^l_j=\pi^l_i$ for any $1\<i\<j\<l$.

Let $I$ be the set of all positive integers. For the definition of an inverse system, the reader
is directed to \cite{Rotman}. Then we have the following theorem.

\begin{theorem}\label{limit1}
$\{Aut_*(H)/Aut_i(H), \pi^j_i\}$ is an inverse system of groups with index set $I$. Moreover,
there is a group isomorphism $Aut_*(H)\cong\underleftarrow{\rm lim}(Aut_*(H)/Aut_i(H))$.
\end{theorem}

\begin{proof}
By the above discussion, $\{Aut_*(H)/Aut_i(H), \pi^j_i\}$ is an inverse system of groups with index set $I$.
It is well-known that the inverse limit $\underleftarrow{\rm lim}(Aut_*(H)/Aut_i(H))$ exists. The group
$G:=\underleftarrow{\rm lim}(Aut_*(H)/Aut_i(H))$ can be described as follows.
Let $\prod_{i\in I}(Aut_*(H)/Aut_i(H))$ be the direct product of the groups $Aut_*(H)/Aut_i(H)$,
and let $p_i: \prod_{i\in I}(Aut_*(H)/Aut_i(H))\ra Aut_*(H)/Aut_i(H)$ be the $i$-th projection for each $i\in I$.
Then
$$G=\{(z_i)_{i\in I}\in\prod\limits_{i\in I}(Aut_*(H)/Aut_i(H))|
\pi^j_i(z_j)=z_i, \mbox{ whenever } i\<j\}.$$
Define $\eta_i: G\ra Aut_*(H)/Aut_i(H)$ as the restriction $p_i|G$. Then by the above discussion,
there is a unique group homomorphism $\rho: Aut_*(H)\ra G$ such that $\eta_i\rho=\pi_i$ for all $i\in I$.
For any $\phi\in Aut_*(H)$,
$$\rho(\phi)=(\pi_i(\phi))_{i\in I}=(\phi Aut_i(H))_{i\in I}.$$
Since $\bigcap_{i\in I}Aut_i(H)=\{\rm id\}$, $\rho$ is injective. Now let
$g=(z_i)_{i\in I}\in G$. Then for any $i\in I$, there is a $\phi_i\in Aut_*(H)$ such that
$z_i=\phi_iAut_i(H)$ in $Aut_*(H)/Aut_i(H)$. Define a linear map $\phi: H\ra H$ by
$\phi(x^n)=x^n$ and $\phi(x^ny^m)=\phi_m(x^ny^m)$ for all $n\in\mathbb Z$ and $m\>1$.
Obviously, $\phi$ is well-defined and independent of the choices of $\phi_i$.
Moreover, for any $n\in\mathbb Z$ and $m\>1$,
$\phi(x^ny^m)=\phi_m(x^ny^m)=x^ny^m+h$ for some $h\in\sum_{l=0}^{m-1}H(l)$ by the definition of
$Aut_*(H)$ and Lemma \ref{firstitem}. It follows that $\phi$ is a $k$-linear automorphism of $H$.
If $1\<i\<j$, then $\phi_iAut_i(H)=z_i=\pi^j_i(z_j)=\pi^j_i(\phi_jAut_j(H))=\pi^j_i\pi_j(\phi_j)
=\pi_i(\phi_j)=\phi_jAut_i(H)$ in $Aut_*(H)/Aut_i(H)$. Hence $\phi(x^n)=x^n=\phi_j(x^n)$
and $\phi(x^ny^i)=\phi_i(x^ny^i)=\phi_j(x^ny^i)$ for all $n\in\mathbb Z$ and $1\<i\<j$. It follows
that $\phi(h)=\phi_j(h)$ for all $h\in\sum_{i=0}^jH(i)$, where $j\>1$. Thus, one can see that
$\phi\in Aut_*(H)$ and $\phi Aut_j(H)=\phi_j Aut_j(H)=z_j$ in $Aut_*(H)/Aut_j(H)$ for all $j\in I$.
This shows that $\rho(\phi)=g$, and so $\rho$ is surjective.
\end{proof}

For any $s\>1$, let $G_s=k^{\mathbb Z}\times\cdots\times k^{\mathbb Z}$ be the Cartesian product set
of $s$-copies of $k^{\mathbb Z}$.
When $s=1$, $G_1=k^{\mathbb Z}$ is an additive group as stated before. From Proposition
\ref{Iso1} and its proof, there is a group isomorphism
$$\Phi_1: G_1\ra Aut_*(H)/Aut_1(H),\ \b^{(1)}\mapsto\phi^{(1)}_{\b^{(1)}}Aut_1(H),$$
where $\b^{(1)}=(\b^{(1)}_n)_{n\in\mathbb Z}\in k^{\mathbb Z}=G_1$ and $\phi^{(1)}_{\b^{(1)}}$
is defined as in the proof of Proposition \ref{Iso1}.
Now let $i>1$. Then one can define a set map
$$\Phi_i: G_i\ra Aut_*(H)/Aut_i(H),\ (\b^{(1)},\b^{(2)},\cdots,\b^{(i)}) \mapsto
\phi^{(1)}_{\b^{(1)}}\phi^{(2)}_{\b^{(2)}}\cdots\phi^{(i)}_{\b^{(i)}}Aut_i(H),$$
where $\b^{(s)}=(\b^{(s)}_n)_{n\in\mathbb Z}\in k^{\mathbb Z}$ for all $1\<s\<i$,
$\phi^{(1)}_{\b^{(1)}}$ is defined as in the proof of Proposition \ref{Iso1}, and $\phi^{(s)}_{\b^{(s)}}$ is defined
as in the proof of Proposition \ref{Isos} for all $2\<s\<i$.

\begin{lemma}\label{strui}
Let $i>1$. Then $\Phi_i: G_i\ra Aut_*(H)/Aut_i(H)$ is a bijection. Consequently,
there exists a unique group structure on $G_i$ such that $\Phi_i$ is a group isomorphism.
\end{lemma}

\begin{proof}
For any $(\b^{(1)},\b^{(2)},\cdots,\b^{(i)})\in G_i$, it follows from the proofs of Propositions \ref{Iso1}
and \ref{Isos} that $\phi^{(s)}_{\b^{(s)}}\in Aut_*(H)$ for $1\<s\<i$, and that
$\phi^{(s)}_{\b^{(s)}}\in Aut_{s-1}(H)$ for $1<s\<i$. Hence $\Phi_i$ is
well-defined, and
$(\phi^{(1)}_{\b^{(1)}}\phi^{(2)}_{\b^{(2)}}\cdots\phi^{(i)}_{\b^{(i)}})(x^ny^s)
=(\phi^{(1)}_{\b^{(1)}}\cdots\phi^{(s)}_{\b^{(s)}})(x^ny^s)$ for all $n\in\mathbb Z$ and
$1\<s\<i$.

Let $(\b^{(1)}, \b^{(2)}, \cdots, \b^{(i)}), (\g^{(1)}, \g^{(2)}, \cdots, \g^{(i)})\in G_i$, and assume that $$\phi^{(1)}_{\b^{(1)}}\phi^{(2)}_{\b^{(2)}}\cdots\phi^{(i)}_{\b^{(i)}}Aut_i(H)
=\phi^{(1)}_{\g^{(1)}}\phi^{(2)}_{\g^{(2)}}\cdots\phi^{(i)}_{\g^{(i)}}Aut_i(H),$$
in $Aut_*(H)/Aut_i(H)$. Then
$\phi^{(1)}_{\b^{(1)}}\phi^{(2)}_{\b^{(2)}}\cdots\phi^{(i)}_{\b^{(i)}}(h)
=\phi^{(1)}_{\g^{(1)}}\phi^{(2)}_{\g^{(2)}}\cdots\phi^{(i)}_{\g^{(i)}}(h)$
for all $h\in\sum_{j=0}^iH(j)$. Hence for all $n\in\mathbb Z$, we have
$$\begin{array}{rcl}
\phi^{(1)}_{\b^{(1)}}(x^ny)&=&\phi^{(1)}_{\b^{(1)}}\phi^{(2)}_{\b^{(2)}}\cdots\phi^{(i)}_{\b^{(i)}}(x^ny)\\
&=&\phi^{(1)}_{\g^{(1)}}\phi^{(2)}_{\g^{(2)}}\cdots\phi^{(i)}_{\g^{(i)}}(x^ny)\\
&=&\phi^{(1)}_{\g^{(1)}}(x^ny).\\
\end{array}$$
That is, $x^ny+\b^{(1)}_n(x^{n+1}-x^n)=x^ny+\g^{(1)}_n(x^{n+1}-x^n)$ for all $n\in\mathbb Z$.
Thus, $\b^{(1)}_n=\g^{(1)}_n$ for all $n\in\mathbb Z$, and so $\b^{(1)}=\g^{(1)}$.
Let $1\<s<i$ and suppose $\b^{(j)}=\g^{(j)}$ for all $1\<j\<s$. Then
$\phi^{(j)}_{\b^{(j)}}=\phi^{(j)}_{\g^{(j)}}$ for all $1\<j\<s$. It follows that
$\phi^{(s+1)}_{\b^{(s+1)}}\cdots\phi^{(i)}_{\b^{(i)}}(h)
=\phi^{(s+1)}_{\g^{(s+1)}}\cdots\phi^{(i)}_{\g^{(i)}}(h)$
for all $h\in\sum_{j=0}^iH(j)$. Hence for all $n\in\mathbb Z$, we have
$$\begin{array}{rcl}
\phi^{(s+1)}_{\b^{(s+1)}}(x^ny^{s+1})&=&\phi^{(s+1)}_{\b^{(s+1)}}\cdots\phi^{(i)}_{\b^{(i)}}(x^ny^{s+1})\\
&=&\phi^{(s+1)}_{\g^{(s+1)}}\cdots\phi^{(i)}_{\g^{(i)}}(x^ny^{s+1})\\
&=&\phi^{(s+1)}_{\g^{(s+1)}}(x^ny^{s+1}).\\
\end{array}$$
That is, $x^ny^{s+1}+\b^{(s+1)}_n(x^{n+s+1}-x^n)=x^ny^{s+1}+\g^{(s+1)}_n(x^{n+s+1}-x^n)$ for all $n\in\mathbb Z$
by Remark \ref{formula}. Thus, $\b^{(s+1)}_n=\g^{(s+1)}_n$ for all $n\in\mathbb Z$, and so $\b^{(s+1)}=\g^{(s+1)}$.
Therefore, $\b^{(j)}=\g^{(j)}$ for all $1\<j\<i$. This shows that $\Phi_i$ is injective.

Now assume $\phi\in Aut_*(H)$. Let $\b^{(1)}=f_1(\phi)\in k^{\mathbb Z}$. Then by Proposition \ref{Iso1} and its proof,
$f_1(\phi^{(1)}_{\b^{(1)}})=f_1(\phi)$, and hence $(\phi^{(1)}_{\b^{(1)}})^{-1}\phi\in Aut_1(H)$.
Then let $\b^{(2)}=f_2((\phi^{(1)}_{\b^{(1)}})^{-1}\phi)\in k^{\mathbb Z}$. Then by Proposition \ref{Isos} and its proof,
$f_2(\phi^{(2)}_{\b^{(2)}})=f_2((\phi^{(1)}_{\b^{(1)}})^{-1}\phi)$,
and hence $(\phi^{(2)}_{\b^{(2)}})^{-1}(\phi^{(1)}_{\b^{(1)}})^{-1}\phi\in Aut_2(H)$.
Let $2\<s<i$ and suppose that we have found
$\b^{(1)}, \cdots, \b^{(s)}\in k^{\mathbb Z}$ such that
$(\phi^{(s)}_{\b^{(s)}})^{-1}\cdots(\phi^{(1)}_{\b^{(1)}})^{-1}\phi\in Aut_s(H)$.
Let $\b^{(s+1)}=f_{s+1}((\phi^{(s)}_{\b^{(s)}})^{-1}\cdots(\phi^{(1)}_{\b^{(1)}})^{-1}\phi)\in k^{\mathbb Z}$.
Then it follows from Proposition \ref{Isos} and its proof that
$f_{s+1}(\phi^{(s+1)}_{\b^{(s+1)}})=f_{s+1}((\phi^{(s)}_{\b^{(s)}})^{-1}\cdots(\phi^{(1)}_{\b^{(1)}})^{-1}\phi)$,
and so $(\phi^{(s+1)}_{\b^{(s+1)}})^{-1}(\phi^{(s)}_{\b^{(s)}})^{-1}\cdots(\phi^{(1)}_{\b^{(1)}})^{-1}\phi\in Aut_{s+1}(H)$.
Thus, we have proved that there are $\b^{(1)}, \b^{(2)}, \cdots, \b^{(i)}\in k^{\mathbb Z}$ such that
$(\phi^{(i)}_{\b^{(i)}})^{-1}\cdots(\phi^{(2)}_{\b^{(2)}})^{-1}(\phi^{(1)}_{\b^{(1)}})^{-1}\phi\in Aut_i(H)$.
Hence $\phi^{(1)}_{\b^{(1)}}\phi^{(2)}_{\b^{(2)}}\cdots \phi^{(i)}_{\b^{(i)}}Aut_i(H)
=\phi Aut_i(H)$, which implies that $\Phi_i$ is surjective.
\end{proof}

As stated before, $G_1$ is an additive group as the direct product $k^{\mathbb Z}$ of
$\mathbb Z$-copies of the additive group $k$.
For $i>1$, the operation of the group $G_i$ can be described as follows.

For $\b=(\b^{(1)}, \b^{(2)}, \cdots, \b^{(i)}), \g=(\g^{(1)}, \g^{(2)}, \cdots, \g^{(i)})\in G_i$,
let $\varphi_{\b}=\phi^{(1)}_{\b^{(1)}}\phi^{(2)}_{\b^{(2)}}\cdots \phi^{(i)}_{\b^{(i)}}$ and
$\varphi_{\g}=\phi^{(1)}_{\g^{(1)}}\phi^{(2)}_{\g^{(2)}}\cdots \phi^{(i)}_{\g^{(i)}}$. Then for any
$n\in\mathbb Z$, we have
$$\varphi_{\b}\varphi_{\g}(x^ny)=\phi^{(1)}_{\b^{(1)}}\phi^{(1)}_{\g^{(1)}}(x^ny)
=x^ny+(\b^{(1)}_n+\g^{(1)}_n)(x^{n+1}-x^n).$$
Let $\d^{(1)}=\b^{(1)}+\g^{(1)}$. Then $(\phi^{(1)}_{\d^{(1)}})^{-1}\varphi_{\b}\varphi_{\g}\in Aut_1(H)$.
Then following the procedure of the proof of Lemma \ref{strui}, one can recursively get
$\d^{(1)}, \d^{(2)}, \cdots, \d^{(i)}\in k^{\mathbb Z}$ such that
$(\phi^{(j)}_{\d^{(j)}})^{-1}\cdots(\phi^{(1)}_{\d^{(1)}})^{-1}\varphi_{\b}\varphi_{\g}\in Aut_j(H)$
for all $1\<j\<i$. Then the multiplication of the group $G_i$ is given by
$$(\b^{(1)}, \b^{(2)}, \cdots, \b^{(i)})(\g^{(1)}, \g^{(2)}, \cdots, \g^{(i)})=
(\d^{(1)}, \d^{(2)}, \cdots, \d^{(i)}).$$
Summarizing the above discussion, one gets the following corollary.

\begin{corollary}\label{producti}
Let $i>1$ and $(\b^{(1)}, \b^{(2)}, \cdots, \b^{(i)}), (\g^{(1)}, \g^{(2)}, \cdots, \g^{(i)})\in G_i$.
The multiplication
$$(\b^{(1)}, \b^{(2)}, \cdots, \b^{(i)})(\g^{(1)}, \g^{(2)}, \cdots, \g^{(i)})=
(\d^{(1)}, \d^{(2)}, \cdots, \d^{(i)})$$
in $G_i$ is determined recursively by $\d^{(1)}=\b^{(1)}+\g^{(1)}$ and
$$\d^{(j)}_n(x^{n+j}-x^n)=\varphi_{\b}\varphi_{\g}(x^ny^{j})
-\phi^{(1)}_{\d^{(1)}}\cdots\phi^{(j-1)}_{\d^{(j-1)}}(x^ny^j)$$
for all $1<j\<i$ and $n\in\mathbb Z$, where $\varphi_{\b}$ and $\varphi_{\g}$ are given as above.
\end{corollary}

%
%

For any $1\<i\<j$, define a map $\psi^j_i: G_j\ra G_i$ by
$$\psi^j_i(\b^{(1)}, \b^{(2)}, \cdots, \b^{(j)})=(\b^{(1)}, \b^{(2)}, \cdots, \b^{(i)}).$$
Then by the above discussion, one can see that $\psi^j_i$ is a group epimorphism
such that the following diagram
$$\begin{array}{ccc}
G_j&\xrightarrow{\Phi_j}&Aut_*(H)/Aut_j(H)\\
\psi^j_i\downarrow\quad &&\downarrow\pi^j_i\\
G_i&\xrightarrow{\Phi_i}&Aut_*(H)/Aut_i(H)\\
\end{array}$$
commutes. Thus, we have the following theorem.

\begin{theorem}\label{limit2}
$\{G_i, \psi^j_i\}$ is an inverse system of groups with index set $I$. Moreover,
there is a group isomorphism $Aut_*(H)\cong\underleftarrow{\rm lim}G_i$.
\end{theorem}

\begin{proof} It follows from Theorem \ref{limit1} and Lemma \ref{strui}.
\end{proof}

From the proof of Theorem \ref{limit1}, the inverse limit $\underleftarrow{\rm lim}G_i$
can be described as follows. Let $\prod_{i\in I}G_i$ be the direct product of the groups
$G_i$, and let $p_i: \prod_{i\in I}G_i\ra G_i$ be the $i$-th projection.
Then $\underleftarrow{\rm lim}G_i$ is a subgroup of $\prod_{i\in I}G_i$ defined by
$$\underleftarrow{\rm lim}G_i=\{(g_i)_{i\in I}\in\prod\limits_{i\in I}G_i|\psi^j_i(g_j)=g_i,
\mbox{ whenever } 1\<i\<j\}.$$
Let $q_i: \underleftarrow{\rm lim}G_i\ra G_i$ be the restriction $p_i|\underleftarrow{\rm lim}G_i$.
Then $\psi^j_iq_j=q_i$ for all $1\<i\<j$.

For any $i\in I$, define a map $\s_i: G_i\ra k^{\mathbb Z}$ by
$\s_i(\b^{(1)}, \cdots, \b^{(i)})=\b^{(i)}$, i.e., $\s_i$ is the $i$-th projection.
Let $G_{\infty}=(k^{\mathbb Z})^I$ be the Cartesian
product set of $I$-copies of $k^{\mathbb Z}$, i.e.,
$$G_{\infty}=\{(\b^{(1)}, \b^{(2)}, \cdots)|\b^{(i)}\in k^{\mathbb Z} \mbox{ for all }i\in I\}.$$
For any $i\in I$, define a map $\tau_i: G_{\infty}\ra G_i$ by
$\tau_i(\b^{(1)}, \b^{(2)}, \cdots)=(\b^{(1)}, \cdots, \b^{(i)})$.
Then $\psi^j_i\tau_j=\tau_i$ for all $1\<i\<j$.

Obviously, there is a bijective map $\tau: G_{\infty}\ra\underleftarrow{\rm lim}G_i$ defined by
$$\tau(z)=(\tau_i(z))_{i\in I}=(\tau_1(z), \tau_2(z), \cdots)$$
for all $z=(z_i)_{i\in I}=(z_1, z_2, \cdots)\in G_{\infty}$. The inverse $\tau^{-1}$ is given by
$$\tau^{-1}(g)=(\s_iq_i(g))_{i\in I}=(\s_i(g_i))_{i\in I}=(\s_1(g_1), \s_2(g_2), \cdots)$$
for all $g=(g_i)_{i\in I}=(g_1, g_2, \cdots)\in\underleftarrow{\rm lim}G_i$.
Note that $q_i\tau=\tau_i$ for all $i\in I$.
Hence there is a unique group structure on $G_{\infty}$ such that $\tau$ is a group isomorphism.
In this case, $\tau_i$ is a group epimorphism for all $i\in I$.
Thus, from Theorem \ref{limit2}, we have proved the following theorem.

\begin{theorem}\label{Iso*}
There is a group isomorphism $Aut_*(H)\cong G_{\infty}$.
\end{theorem}

Explicitly, we can give a group isomorphism $\Phi: G_{\infty}\ra Aut_*(H)$ as in the proof of Lemma \ref{strui}.
For $(\b^{(1)}, \b^{(2)}, \cdots)\in G_{\infty}$, $\phi:=\Phi(\b^{(1)}, \b^{(2)}, \cdots)$ is given by
$$\phi(x^n)=x^n, \ \phi(x^ny^m)=\phi^{(1)}_{\b^{(1)}}\cdots\phi^{(m)}_{\b^{(m)}}(x^ny^m)$$
for all $n\in\mathbb Z$ and $m\>1$. Conversely, for $\phi\in Aut_*(H)$,
$(\b^{(1)}, \b^{(2)}, \cdots):=\Phi^{-1}(\phi)$ is given recursively by
$\b^{(1)}_n(x^{n+1}-x^n)=\phi(x^ny)-x^ny$ and
$$\b^{(i)}_n(x^{n+i}-x^n)=\phi(x^ny^i)-\phi^{(1)}_{\b^{(1)}}\cdots\phi^{(i-1)}_{\b^{(i-1)}}(x^ny^i)$$
for $i>1$, where $n\in\mathbb Z$.

\begin{remark}
By Corollary \ref{producti}, the multiplication of the group $G_{\infty}$ can be described as follows:
for $(\b^{(1)}, \b^{(2)}, \cdots), (\g^{(1)}, \g^{(2)}, \cdots)\in G_{\infty}$,
the multiplication
$$(\b^{(1)}, \b^{(2)}, \cdots)(\g^{(1)}, \g^{(2)}, \cdots)=
(\d^{(1)}, \d^{(2)}, \cdots)$$
is determined recursively by $\d^{(1)}=\b^{(1)}+\g^{(1)}$ and
$$\d^{(i)}_n(x^{n+i}-x^n)=\phi^{(1)}_{\b^{(1)}}\cdots\phi^{(i)}_{\b^{(i)}}
\phi^{(1)}_{\g^{(1)}}\cdots\phi^{(i)}_{\g^{(i)}}(x^ny^{i})
-\phi^{(1)}_{\d^{(1)}}\cdots\phi^{(i-1)}_{\d^{(i-1)}}(x^ny^i)$$
for all $i>1$ and $n\in\mathbb Z$.

By a tedious computation, one can get the formula of $\d^{(2)}$ and $\d^{(3)}$ as follows:.
$$\begin{array}{ccl}
\d^{(2)}&=&b^{(2)}+\g^{(2)}-(2)_q\b^{(1)}\g^{(1)}[1]),\\
\d^{(3)}&=&b^{(3)}+\g^{(3)}-(3)_q(\b^{(2)}\g^{(1)}[1]-\b^{(2)}[1]\g^{(1)})\\
&&-(3)!_q(\b^{(1)}+\g^{(1)})\b^{(1)}[1]\g^{(1)}[2].\\
\end{array}$$
Thus, one knows the formulas of the multiplications of $G_2$ and $G_3$. However,
it is difficult to give the explicit formulae of $\d^{(i)}$ for larger $i$.
\end{remark}

\begin{lemma}\label{action1}
Let $\a=(\a_n)_{n\in\mathbb Z}\in(k^{\times})^{\mathbb Z}$, $\b=(\b_n)_{n\in\mathbb Z}\in k^{\mathbb Z}$ and $r\in\mathbb Z$. Then
$(\phi_{\a}\theta_r)\phi^{(1)}_{\b}(\phi_{\a}\theta_r)^{-1}=\phi^{(1)}_{\a^{-1}\b[-r]}$. In particular,
$\phi_{\a}\phi^{(1)}_{\b}\phi_{\a}^{-1}=\phi^{(1)}_{\a^{-1}\b}$.
\end{lemma}

\begin{proof}
For any $n\in\mathbb Z$ and $0\<l\<m$, define $a_{n,m,l}$ and $b_{n,m,l}$ in $H_0$ by $a_{n,m,m}=b_{n,m,m}=x^n$ and
\begin{align*}
  a_{n,m,l} =& (m,m-l)_q (\b_{n,m-l}x^{n+m-l}-\b_{n,m-l-1}\b_{n+m-1}x^{n+m-l-1}) \\
  b_{n,m,l} =&(m,m-l)_q ((\a^{-1}\b[-r])_{n,m-l}x^{n+m-l}\\
  &-(\a^{-1}\b[-r])_{n,m-l-1}(\a^{-1}\b[-r])_{n+m-1}x^{n+m-l-1})
\end{align*}
for $0\<l<m$. Then by the proof of Proposition \ref{Iso1}, we have
$$\phi^{(1)}_{\b}(x^ny^m)=\sum\limits_{l=0}^ma_{n,m,l}y^l\quad\mbox{ and }\quad
\phi^{(1)}_{\a^{-1}\b[-r]}(x^ny^m)=\sum\limits_{l=0}^mb_{n,m,l}y^l$$
for any $n\in\mathbb Z$ and $m\>0$. By a straightforward verification, one can check that
$$\a^{-1}_{n,m}\a_{n+m-l, l}\b_{n-r, m-l}=(\a^{-1}\b[-r])_{n,m-l}$$
and
$$\a^{-1}_{n,m}\a_{n+m-l-1, l}\b_{n-r, m-l-1}\b_{n-r+m-1}=(\a^{-1}\b[-r])_{n,m-l-1}(\a^{-1}\b[-r])_{n+m-1}$$
for all $n\in\mathbb Z$ and $0\<l<m$. Hence we have
$$\begin{array}{rcl}
\a^{-1}_{n,m}\phi_{\a}\theta_r(\b_{n-r, m-l}x^{n-r+m-l}y^l)
&=&\a^{-1}_{n,m}\phi_{\a}(\b_{n-r, m-l}x^{n+m-l}y^l)\\
&=&\a^{-1}_{n,m}\a_{n+m-l, l}\b_{n-r, m-l}x^{n+m-l}y^l\\
&=&(\a^{-1}\b[-r])_{n,m-l}x^{n+m-l}y^l\\
\end{array}$$
and
$$\begin{array}{cl}
&\a^{-1}_{n,m}\phi_{\a}\theta_r(\b_{n-r, m-l-1}\b_{n-r+m-1}x^{n-r+m-l-1}y^l)\\
=&\a^{-1}_{n,m}\phi_{\a}(\b_{n-r, m-l-1}\b_{n-r+m-1}x^{n+m-l-1}y^l)\\
=&\a^{-1}_{n,m}\a_{n+m-l-1,l}\b_{n-r, m-l-1}\b_{n-r+m-1}x^{n+m-l-1}y^l\\
=&(\a^{-1}\b[-r])_{n,m-l-1}(\a^{-1}\b[-r])_{n+m-1}x^{n+m-l-1}y^l,\\
\end{array}$$
and consequently, $\a^{-1}_{n,m}\phi_{\a}\theta_r(a_{n-r, m, l}y^l)=b_{n,m,l}y^l$
for all $n\in\mathbb Z$ and $0\<l<m$. It follows that
$$\begin{array}{cl}
&(\phi_{\a}\theta_r)\phi^{(1)}_{\b}(\phi_{\a}\theta_r)^{-1}(x^ny^m)\\
=&\phi_{\a}\theta_r\phi^{(1)}_{\b}\theta_r^{-1}(\a_{n, m}^{-1}x^ny^m)\\
=&\phi_{\a}\theta_r\phi^{(1)}_{\b}(\a_{n, m}^{-1}x^{n-r}y^m)\\
=&\a_{n, m}^{-1}\phi_{\a}\theta_r(\sum\limits_{l=0}^ma_{n-r, m, l}y^l)\\
=&\a_{n, m}^{-1}\phi_{\a}\theta_r(x^{n-r}y^m)+
\a_{n, m}^{-1}\phi_{\a}\theta_r(\sum\limits_{0\<l<m}a_{n-r, m, l}y^l)\\
=&x^{n}y^m+\sum\limits_{0\<l<m}b_{n, m, l}y^l\\
=&\phi^{(1)}_{\a^{-1}\b[-r]}(x^ny^m)
\end{array}$$
for $n\in\mathbb Z$ and $m\>0$. This completes the proof.
\end{proof}

\begin{lemma}\label{actions}
Let $\a=(\a_n)_{n\in\mathbb Z}\in(k^{\times})^{\mathbb Z}$, $\b=(\b_n)_{n\in\mathbb Z}\in k^{\mathbb Z}$,
$r\in\mathbb Z$ and $s>1$. Then
$$(\phi_{\a}\theta_r)\phi^{(s)}_{\b}(\phi_{\a}\theta_r)^{-1}=\phi^{(s)}_{\a^{-1}\a^{-1}[1]\cdots\a^{-1}[s-1]\b[-r]}$$
In particular, $\phi_{\a}\phi^{(s)}_{\b}\phi_{\a}^{-1}=\phi^{(s)}_{\a^{-1}\a^{-1}[1]\cdots\a^{-1}[s-1]\b}$.
\end{lemma}

\begin{proof}
It is similar to Lemma \ref{action1} by using Proposition \ref{Isos} and Remark \ref{formula}.
\end{proof}

For any $\a\in(k^{\times})^{\mathbb Z}$ (or $k^{\mathbb Z}$) and $i\>1$, define $\a\langle i\rangle$ by $\a\langle 1\rangle=\a$
and $\a\langle i\rangle=\a\a[1]\cdots\a[i-1]$ for $i>1$. Then $(\a\b)\langle i\rangle=\a\langle i\rangle\b\langle i\rangle$
and $\a[r]\langle i\rangle=\a\langle i\rangle[r]$ for all $\a, \b\in(k^{\times})^{\mathbb Z}$ (or $k^{\mathbb Z}$),
$r\in\mathbb Z$ and $i\>1$.

For any $(\a, r)\in(k^{\times})^{\mathbb Z}\rtimes\mathbb{Z}$
and $(\b^{(i)})_{i\in I}\in G_{\infty}$, define $(\a, r)\cdot(\b^{(i)})_{i\in I}\in G_{\infty}$ by
$$(\a, r)\cdot(\b^{(i)})_{i\in I}=(\a^{-1}\langle i\rangle\b^{(i)}[-r])_{i\in I}
=(\a^{-1}\langle 1\rangle\b^{(1)}[-r], \a^{-1}\langle 2\rangle\b^{(2)}[-r], \cdots).$$
Then one can check that this gives rise to a group action of
$(k^{\times})^{\mathbb Z}\rtimes\mathbb{Z}$ on $G_{\infty}$. Hence one can form a semidirect product
group $G_{\infty}\rtimes((k^{\times})^{\mathbb Z}\rtimes\mathbb{Z})$.

Since $(k^{\times})^{\mathbb Z}$ can be embedded
into $(k^{\times})^{\mathbb Z}\rtimes\mathbb{Z}$ as a (normal) subgroup, there is a group action of
$(k^{\times})^{\mathbb Z}$ on $G_{\infty}$ given by
$$\a\cdot(\b^{(i)})_{i\in I}=(\a^{-1}\langle i\rangle\b^{(i)})_{i\in I}
=(\a^{-1}\langle 1\rangle\b^{(1)}, \a^{-1}\langle 2\rangle\b^{(2)}, \cdots),$$
where $\a\in (k^{\times})^{\mathbb Z}$ and $(\b^{(i)})_{i\in I}\in G_{\infty}$. Thus, one can form another
semidirect product group $G_{\infty}\rtimes(k^{\times})^{\mathbb Z}$. Now we can state the main theorem of this section.

\begin{theorem}
There are group isomorphisms $Aut_c(H)\cong G_{\infty}\rtimes((k^{\times})^{\mathbb Z}\rtimes\mathbb{Z})$
and $Aut_0(H)\cong G_{\infty}\rtimes(k^{\times})^{\mathbb Z}$. Explicitly, the group isomorphisms are given by
$$G_{\infty}\rtimes((k^{\times})^{\mathbb Z}\rtimes\mathbb{Z})\ra Aut_c(H),\quad
(g, (\a, r))\mapsto \Phi(g)\phi_{\a}\theta_r$$
and
$$G_{\infty}\rtimes(k^{\times})^{\mathbb Z}\ra Aut_0(H),\quad
(g, \a)\mapsto \Phi(g)\phi_{\a},$$
where $\Phi$ is the group isomorphism from $G_{\infty}$ to $Aut_*(H)$ given before.
\end{theorem}

\begin{proof}
It follows from Lemma \ref{Semidirect}, Theorems \ref{iso-gr-0}, \ref{iso-gr} and
\ref{Iso*}, and Lemmas \ref{action1} and \ref{actions}.
\end{proof}

\begin{remark}
For the hypothesis that $q$ is not a root of unity, we only need the condition
that $\left(
        \begin{array}{c}
          n \\
          i \\
        \end{array}
      \right)_q\neq 0$ for all $0\<i\<n$. Thus, if ${\rm char}(k)=0$ and $q=1$, then all the arguments
in the paper are still valid. Hence all the results of the paper hold too in this case.
\end{remark}

\section*{ACKNOWLEDGMENTS}
\hskip\parindent

This work is supported by NSF of China (No. 11171291).\\

\end{document}